\documentclass[12pt]{amsart}

\newtheorem{theorem}{Theorem}[section] 
\newtheorem{corollary}[theorem]{Corollary} 
 
\newtheorem{lemma}[theorem]{Lemma} 
\newtheorem{proposition}[theorem]{Proposition}

\theoremstyle{definition}
\newtheorem{definition}{Definition}

\theoremstyle{remark}

  {\end{list}}

\usepackage{amssymb} 
\usepackage{verbatim}

\newcommand{\ord}{\operatorname{ord}}

\renewcommand{\deg}{{\operatorname{deg}}}

\newcommand{\GL}{\operatorname{GL}}
\newcommand{\PGL}{\operatorname{PGL}}

\newcommand{\Res}{\operatorname{Res}}

\newcommand{\Ann}{\operatorname{Ann}}

\newcommand{\diam}{\operatorname{diam}}

\newcommand{\Lip}{\operatorname{Lip}}
\newcommand{\BPP}{\operatorname{{\bf P}}}
\newcommand{\BHH}{\operatorname{{\bf H}}}
\newcommand{\BAA}{\operatorname{{\bf A}}}

\def\ord{{\mathop{\rm ord}}}
\def\ordRes{{\mathop{\rm ordRes}}}

\def\deg{{\mathop{\rm deg}}}

\def\GL{{\mathop{\rm GL}}}
\def\PGL{{\mathop{\rm PGL}}}

\def\Berk{{\mathop{\rm Berk}}}
\def\Res{{\mathop{\rm Res}}}

\def\GPR{{\mathop{\rm GPR}}}
\def\RP{{\mathop{\rm RP}}}
\def\GIR{{\mathop{\rm GIR}}}
\def\Rat{{\mathop{\rm Rat}}}

\def\vv{{\vec{v}}}
\def\vw{{\vec{w}}}

\def\v1{{\vec{1}}}
\def\vbb1{\vec{{\mathbf 1}}}

\def\BF1{{\mathbf 1}}

\def\AA{{\mathbb A}}

\def\CC{{\mathbb C}}

\def\HH{{\mathbb H}}

\def\PP{{\mathbb P}}

\def\RR{{\mathbb R}}

\def\ZZ{{\mathbb Z}}

\def\cB{{\mathcal B}}

\def\cD{{\mathcal D}}

\def\cI{{\mathcal I}}

\def\cO{{\mathcal O}}

\def\hbar{{\overline{h}}}

\def\fM{{\mathfrak m}}

\def\tk{{\widetilde{k}}}

\def\tA{{\widetilde{A}}}
\def\tB{{\widetilde{B}}}

\def\tF{{\widetilde{F}}}
\def\tG{{\widetilde{G}}}

\def\tX{{\widetilde{X}}}
\def\tY{{\widetilde{Y}}}

\def\tphi{{\widetilde{\varphi}}}

\def\Berk{{\rm Berk}}

\def\<{{\langle }}
\def\>{{\rangle }}

\def\<<{{\langle \! \langle}}
\def\>>{{\rangle \! \rangle}} 

\def\({(\!(}
\def\){)\!)}

\def\[{[\![}
\def\]{]\!]}

\DeclareMathSymbol{\varnothing} {\mathord}{AMSb}{"3F} 
 
\theoremstyle{definition} 
\theoremstyle{remark} 

\begin{document}
\title{The Lipschitz Constant of a Nonarchimedean Rational Function}

\author{Robert Rumely}
\address{Robert Rumely\\ 
Department of Mathematics\\
University of Georgia\\
Athens, Georgia 30602\\
USA}
\email{rr@math.uga.edu}

\author{Stephen Winburn}
\address{Stephen Winburn\\ 
Ally Corporation\\
440 S. Church Street\\
Charlotte, N.C. 28202\\
USA}
\email{Stephen.Winburn@ally.com}

\date{December 3, 2015}

\begin{abstract}
Let $K$ be a complete, algebraically closed nonarchimedean valued field, 
and let $\varphi(z) \in K(z)$ have degree $d \ge 1$.  
We provide explicit bounds for the Lipschitz constants $\Lip_\Berk(\varphi)$, $\Lip_{\PP^1(K)}(\varphi)$,  
in terms of algebraic and geometric invariants of $\varphi$.
\end{abstract} 

\maketitle

Let $K$ be a complete, algebraically closed nonarchimedean valued field with absolute value $| \cdot |$
and associated valuation $\ord(\cdot)$. 
Write $\cO$ for the ring of integers of $K$, $\cO^\times$ for its group of units, 
$\fM$ for its maximal ideal, and $\tk$ for its residue field.

Let $\varphi(z)  \in K(z)$  be a rational function with $\deg(\varphi) = d \ge 1$.  
The action of $\varphi$ on $\PP^1(K)$ 
extends canonically to an action on Berkovich projective line $\BPP^1_K$.   
In \cite{F-RL2} Favre and Rivera-Letelier define a metric $d(x,y)$ on $\BPP^1_K$, which induces the 
strong topology on $\BPP^1_K$;  if $\|x,y\|$ denotes the spherical metric on $\PP^1(K)$, then 
$d(x,y)$ restricts to $2 \|x,y\|$ on $\PP^1(K)$.  
The Lipschitz constant of $\varphi$ with respect to $d(x,y)$ is  
\begin{equation*}
\Lip_\Berk(\varphi) \ := \ \sup_{\substack{ x,y \in\, \BPP^1_K \\ x \ne y }} \frac{d(\varphi(x),\varphi(y))}{d(x,y)} \ .
\end{equation*}
It is the only inexplicit term in Favre and Rivera-Letelier's quantitative equidistribution theorem 
for dynamical small points (\cite{F-RL2}, Theorem 7). 
One may also be interested in the Lipschitz constant of $\varphi$ on classical points,  
\begin{equation*}
\Lip_{\PP^1(K)}(\varphi) 
\ := \ \sup_{\substack{ x,y \in \PP^1(K) \\ x \ne y }} \frac{\|\varphi(x),\varphi(y)\|}{\|x,y\|} \ .
\end{equation*} 

The purpose of this paper is to bound $\Lip_{\PP^1(K)}(\varphi)$ and $\Lip_\Berk(\varphi)$ 
using algebraic and geometric $\GL_2(\cO)$-invariants of $\varphi$. Let $(F,G)$
be a {\em normalized representation} for $\varphi$, a pair of homogeneous polynomials $F(X,Y), G(X,Y) \in \cO[X,Y]$ 
of degree $d$, with at least one coefficient of $F$ or $G$ in $\cO^{\times}$, 
such that $[F:G]$ gives the action of $\varphi$ on $\PP^1(K)$.  
The absolute value of the resultant $|\Res(\varphi)| := |\Res(F,G)|$
is independent of the choice of normalized representation.  

\begin{theorem} \label{FirstCor} 
Let $K$ be a complete, algebraically closed nonarchimedean field,  
and let $\varphi(z) \in K(z)$ have degree $d \ge 1$.  Then
\begin{equation} \label{FirstCorBound} 
\Lip_{\PP^1(K)}(\varphi)  \le  \frac{1}{|\Res(\varphi)|} \ , \quad 
\Lip_\Berk(\varphi)  \le  \max\Big( \frac{d}{|\Res(\varphi)|} \, , \frac{1}{|\Res(\varphi)|^d} \, \Big) 
\ .                               
\end{equation} 
\end{theorem}

\noindent{In particular,} $\Lip_{\PP^1(K)}(\varphi)$ and $\Lip_\Berk(\varphi)$ are uniformly bounded
in terms of the proximity of $\varphi$ to the boundary of the parameter space $\Rat_d$.

\smallskip
We prove Theorem \ref{FirstCor} by first bounding $\Lip_\Berk(\varphi)$ and $\Lip_{\PP^1(K)}(\varphi)$ 
in terms of geometric invariants of $\varphi$. Suppose 
\begin{equation} \label{varphi_Factorization}
\varphi(z) \ = \ C \cdot \frac{\prod_{i=1}^N (z-\alpha_i)}{\prod_{j=1}^M (z-\beta_j)} \ , 
\end{equation} 
where $d = \deg(\varphi) = \max(M,N)$.  

If the $\alpha_i, \beta_j$ are fixed but $|C| \rightarrow \infty$, then 
$\Lip_{\PP^1(K)}(\varphi), \Lip_\Berk(\varphi) \rightarrow \infty$.  
We introduce the {\em Gauss Image Radius} $\GIR(\varphi)$ 
as a geometric replacement for $C$:  given $x \in \BPP^1_K$, let $0 \le \diam_G(x) \le 1$ 
be its diameter with respect to the Gauss point $\zeta_G$  (see \S1),  
and put  
\begin{equation*}
\GIR(\varphi) \ := \ \diam_G(\varphi(\zeta_G)) \ .
\end{equation*} 
Likewise, if $C$ is fixed, but a root approaches a pole, then 
$\Lip_{\PP^1(K)}(\varphi), \Lip_\Berk(\varphi) \rightarrow \infty$. 
Let the Root-Pole number $\RP(\varphi)$ be the minimal spherical distance between a zero and a pole of $\varphi$. 
We introduce the {\em Ball-Mapping radius} $B_0(\varphi)$ as a geometric replacement for $\RP(\varphi)$:  
for each $a \in \PP^1(K)$ and each $0 < r \le 1$, let $B(a,r)^- = \{z \in \PP^1(K) : \|z,a\| < r\}$,
and let $\cB(a,r)^-$ be the smallest open connected subset of $\BPP^1_K$ containing $B(a,r)^-$.  It is known that
$\varphi(\cB(a,r)^-)$ is either an open ball $\cB_Q(\vv)^-$, or is all of $\BPP^1_K$.  Define 
\begin{equation*}
B_0(\varphi) 
\ := \ \sup \{ \ 0 < r \le 1 \ : \text{ for all $a \in \PP^1(K)$, $\varphi(\cB(a,r)^-)$ is a ball } \} \ .
\end{equation*}

The number $\GIR(\varphi)$ can be readily computed from the coefficients of $\varphi$
(see Proposition \ref{GIRProp}).   We do not know how to determine  $B_0(\varphi)$ in general;
this seems an interesting problem.  
By Proposition \ref{GaussInequalitiesProp} and Corollary \ref{ResultantBoundCor}
one has $B_0(\varphi) \ge \RP(\varphi) \ge |\Res(\varphi)|$.
In Proposition \ref{B_0RationalityProp} 
we show that $B_0(\varphi) \in |K^{\times}|$ 
and is achieved by some ball $\cB(a,r)^-$.  

Our main result is 

\begin{theorem}  \label{MainThm0}  Let $K$ be a complete, algebraically closed nonarchimedean field,  
and let $\varphi(z) \in K(z)$ have degree $d \ge 1$.  Then
\begin{equation} \label{MainBound0} 
\Lip_\Berk(\varphi) \ \le \ \max\Big( \frac{1}{\GIR(\varphi) \cdot B_0(\varphi)^d} \, , 
                                  \frac{d}{\GIR(\varphi)^{1/d} \cdot B_0(\varphi)} \Big) \ .
\end{equation} 
\end{theorem} 


Combined with the inequality $B_0(\varphi) \ge \RP(\varphi)$, Theorem \ref{MainThm0} implies the following bounds, 
which may be useful when a factorization of $\varphi$ in the form 
(\ref{varphi_Factorization}) is known:
\begin{equation} \label{SecondCorBound} 
\Lip_\Berk(\varphi) \ \le \ \max\Big( \frac{1}{\GIR(\varphi) \cdot \RP(\varphi)^d} \, , 
                                  \frac{d}{\GIR(\varphi)^{1/d} \cdot \RP(\varphi)} \Big) 
                           \ \le \ \frac{d}{\GIR(\varphi) \cdot \RP(\varphi)^d} \ .
\end{equation}  

The bound in Theorem \ref{MainThm0} 
is sharp when $d=1$, that is, when $\varphi(z)$ is a linear fractional transformation (see Theorem \ref{MobiusCase} below).  
When $d \ge 2$, 
for each triple $\big(d,\GIR(\varphi),B_0(\varphi)\big)$ 
we give examples where $\Lip_\Berk(\varphi)$ is within a factor $(d-1)/d$  of the right side of (\ref{MainBound0})
(see \S\ref{ExamplesSection}). 

\smallskip
Trivially $\Lip_{\PP^1(K)}(\varphi) \le \Lip_\Berk(\varphi)$, however, there is an explicit formula
for $\Lip_{\PP^1(K)}(\varphi)$ which yields a better bound.  
The set $\varphi^{-1}(\{\zeta_G\}) \subset \BPP^1_K \backslash \PP^1(K)$ is finite;     
define the {\em Gauss Pre-Image radius} 
\begin{equation*}
\GPR(\varphi) \ = \ \min_{\varphi(\xi) = \zeta_G}  \diam_G(\xi) \ .
\end{equation*}

\begin{theorem} \label{Classical_Lip} 
Let $K$ be a complete, algebraically closed nonarchimedean field,  
and let $\varphi(z) \in K(z)$ have degree $d \ge 1$. Then
\begin{equation*}
\Lip_{\PP^1(K)}(\varphi) \ = \ \frac{1}{\GPR(\varphi)} \ . 
\end{equation*} 
\end{theorem} 

In Corollary \ref{ResultantBoundCor}, 
we show that $\GIR(\varphi)^d \cdot B_0(\varphi) \ge |\Res(\varphi)|$, and that $\GPR(\varphi) \ge |\Res(\varphi)|$. 
Combining this with Theorems \ref{MainThm0} and \ref{Classical_Lip} yields Theorem \ref{FirstCor}.

\smallskip
For a M\"obius transformation, 
the bounds in Theorems $\ref{FirstCor}$, $\ref{MainThm0}$, and $\ref{Classical_Lip}$ are sharp,
and can be made much more explicit:

\begin{theorem} \label{MobiusCase}
Let $K$ be a complete, algebraically closed nonarchimedean field,  
and let $\varphi(z) = (az+b)/(cz+d) \in K(z)$ have degree $1$. Then $B_0(\varphi) = 1$, and 
\begin{eqnarray*}
\Lip_\Berk(\varphi) & = & \Lip_{\PP^1(K)}(\varphi) \ = \ \frac{1}{\GIR(\varphi)} \ = \ \frac{1}{\GPR(\varphi)} \\
                    & = & \frac{1}{|\Res(\varphi)|} \ = \ \frac{\max(\,|a|,|b|,|c|,|d|\,)}{|ad-bc|} \ .
\end{eqnarray*}
\end{theorem}

The plan of the paper is as follows.  In $\S1$ we recall facts and notation concerning the Berkovich projective line.
In $\S2$ we establish some preliminary lemmas, showing that to bound $\Lip_\Berk(\varphi)$ it suffices to bound it on 
a restricted class of segments $[x,y]$.  In $\S3$ we prove Theorem \ref{MobiusCase}.  
In $\S4$ we study the constants $\GIR(\varphi)$, $B_0(\varphi)$, $\RP(\varphi)$, 
and $\GPR(\varphi)$, and we give a formula for $|\Res(\varphi)|$ 
which may be of independent interest.  
In $\S5$ we prove Theorems \ref{FirstCor}, \ref{MainThm0} and \ref{Classical_Lip}.  
Finally, in $\S6$ we provide examples showing that Theorem \ref{MainThm0} cannot be significantly improved.

We thank Xander Faber and Kenneth Jacobs for useful discussions. In particular we thank Jacobs 
for pointing out that Theorem \ref{MainThm0} 
could yield bounds of the form (\ref{FirstCorBound}).

\section{The Berkovich Projective Line} \label{PreliminariesSection}

The Berkovich projective line over $K$ is a locally ringed space, functorially constructed from $\PP^1/K$, 
whose sheaf of rings comes from rigid analysis  
and whose underlying point set is gotten by gluing the Gel'fand spectra of those rings (see \cite{Ber}). 
We will write $\BPP^1_K$ for its point set, which 
is a uniquely path-connected Hausdorff space containing $\PP^1(K)$.
For proofs of the properties of $\BPP^1_K$ discussed below, see (\cite{B-R}, Chapters 1 and 2); 
for additional facts about $\BPP^1_K$, see   
(\cite{BIJL}, \cite{Ber}, \cite{Fab}, \cite{F-RL2}, \cite{FRLErgodic}, \cite{R-L1}). 

Berkovich's classification theorem (see (\cite{Ber}, p.18), or (\cite{B-R}, p.5)) provides an  
elementary model for\, $\BPP^1_K$:
its points correspond to discs $D(a,r) = \{z \in K : |z-a| \le r\}$, where $a \in K$ and $0 \le r \in \RR$, 
or to cofinal equivalence classes of sequences of nested discs with empty intersection, or to the point $\infty \in \PP^1(K)$. 
There are four kinds of points.  
Type I points, which are the points of $\PP^1(K)$, correspond to degenerate discs of radius $0$ in $K$ 
and the point $\infty \in \PP^1(K)$. 
Type II points correspond to discs $D(a,r)$ with $r$ in the value group $|K^{\times}|$,    
and type III points correspond to discs $D(a,r)$ with $r \notin |K^{\times}|$. 
Type IV points correspond to (cofinal equivalence classes of)
sequences of nested discs with empty intersection; they serve to complete $\BPP^1_K$ but rarely need to be 
dealt with explicitly: they are usually handled by continuity arguments. 

We call $\BAA^1_K = \BPP^1_K \backslash \{\infty\}$ the Berkovich Affine Line.

We write $\zeta_{a,r}$ for the point corresponding to $D(a,r)$. 
The point $\zeta_G := \zeta_{0,1}$ corresponding to $D(0,1)$ is called the {\em Gauss point},
and plays a particularly important role.   

Paths in $\BPP^1_K$ correspond to ascending or descending chains of discs, 
or concatenations of such chains sharing an endpoint.  
For example the path from $0$ to $1$
in $\BPP^1_K$ corresponds to the concatenation of the chains 
$\{D(0,r) : 0 \le r \le 1\}$ and $\{D(1,r) : 0 \le r \le 1\}$;  here $D(0,1) = D(1,1)$. 
The point $\infty$, and type IV points, can also be endpoints of chains.
Topologically, $\BPP^1_K$ is a tree:  for any two points $x, y \in \BPP^1_K$,    
there is a unique path $[x,y]$ between $x$ and $y$.  We will write $(x,y)$, $[x,y)$, and $(x,y]$
for the corresponding open or half-open paths. 

Fix a point $\xi \in \BPP^1_K$. The fact that $\BPP^1_K$ is uniquely path connected 
means that for any two points $x, y \in \BPP^1_K$, we can define the join $x \vee_\xi y$ 
to be the first point where the paths $[x,\xi]$ and $[y,\xi]$ meet.  We will be particularly interested
in the cases where $\xi = \infty$ and $\xi = \zeta_G$;  we denote the corresponding joins by $x \vee_\infty y$
and $x \vee_G y$.  

The fact that $\BPP^1_K$ is uniquely path-connected also means that for each $Q \in \BPP^1_K$, 
the path-components of $\BPP^1_K \backslash \{Q\}$ have the property that for any 
$P_1, P_2$ in the same component, the paths $[Q,P_1]$ and $[Q,P_2]$ share an initial segment.  
Because of this, the components are in $1-1$ correspondence with germs of paths emanating from $Q$, 
which we call {\em tangent vectors} $\vv$ at $Q$.  We write $T_Q$ for the set of tangent vectors at $Q$. 
If $Q$ is of type I or type IV, then $T_Q$ has one element;  if $Q$ is of type III, $T_Q$ has two elements.
If $Q$ is of type II, then $T_Q$ is in $1-1$ correspondence with $\PP^1(\tk)$, for the residue field $\tk = \cO/\fM$. 

\begin{definition} \label{BallDef}    
For each $\vv \in T_Q$, we write $\cB_Q(\vv)^-$ for the component of $\BPP^1_K \backslash \{Q\}$ 
containing points for which $[Q,P]$ prolongs $\vv$. If $P \in \cB_Q(\vv)^-$, we say that $P$
is in the direction $\vv$ at $Q$.  Given $P \ne Q, $
we write $\vv_P$ for the direction in $T_Q$ such that $P \in B_Q(\vv_P)^-$. 
\end{definition} 

Let $\ord(x)$ be the additive valuation on $K$ corresponding to $|x|$; 
there is a unique base $q > 1$ for which $\ord(x) = -\log_q(|x|)$.  We will write $\log(z)$ for $\log_q(z)$.
The set $\BHH^1_K = \BPP^1_K \backslash \PP^1(K)$ is called the {\em Berkovich upper halfspace}; 
it carries a metric $\rho(x,y)$ called the {\em logarithmic path distance}, for which the length of the 
path corresponding to the chain of discs $\{D(a,r) : R_1 \le r \le R_2\}$ is $\log(R_2/R_1)$.

There are two natural topologies on $\BPP^1_K$, called the {\em weak}\, and {\em strong} topologies. 
The basic open sets for the weak topology are the path-components 
of $\BPP^1_K \backslash \{P_1, \ldots, P_n\}$
as $\{P_1, \ldots, P_n\}$ ranges over finite subsets of $\BHH^1_K$. 
Under the weak topology,  $\BPP^1_K$ is compact, and $\PP^1(K)$ is dense in it.  
The weak topology is not in general defined by a metric.
The basic open sets for the strong topology 
are the $\rho(x,y)$-balls 
\begin{equation*}
\cB_\rho(x,r)^- \ = \  \{z \in \BHH^1_K : \rho(x,z) < r\} 
\end{equation*} 
for $x \in \BHH^1_K$ and $r > 0$, together with the basic open sets from the weak topology. 
Under the strong topology, $\BPP^1_K$ is complete but not compact. 
The strong topology is induced by the Favre-Rivera Letelier metric $d(x,y)$, defined at (\ref{FRLdMetricDef}) below.     
Type II points are dense in $\BPP^1_K$ for both topologies.

For each $\varphi(z) \in K(z)$, the action of $\varphi(z)$ on $\PP^1(K)$ extends functorially to an action 
on $\BPP^1_K$ which  is continuous for both the weak and strong topologies.  
If $\varphi$ is nonconstant, the action is open, surjective, and preserves the type of each point.  
The action of $\varphi(z) \in K(z)$ on type II points 
corresponds to its `generic' action on punctured discs, 
in the following sense.  Let $D(a,r)^- = \{z \in K: |z-a| < r\}$. 
One has $\varphi(\zeta_{a,r}) = \zeta_{b,R}$ if and only if there are finitely many points $a_1, \ldots, a_m \in D(a,r)$
and finitely many points $b_1, \ldots, b_m \in D(b,R)$ such that  
\begin{equation*}
\varphi\big(D(a,r) \backslash \bigcup_{i=1}^m D(a_i,r)^-\big) \ = \ D(b,R) \backslash \bigcup_{j=1}^n D(b_j,R)^-   
\end{equation*}    
(see \cite{B-R}, Proposition 2.8).  For example, the inversion map $\iota(z) = 1/z$ satisfies 
\begin{equation*}
\left\{ \begin{array}{ll} 
            \iota(D(0,r) \backslash D(0,r)^-) \ = \ D(0,1/r) \backslash D(0,1/r)^- \ , & \text{ }\\
            \iota(D(a,r) \backslash D(a,r)^-) \ = \ D(1/a,r/|a|^2) \backslash D(1/a,r/|a|^2)^- 
                                      & \text{if $|a| > r$ \ .}
                                 \end{array} \right.
\end{equation*} 
Since a disc can be written as $D(0,r)$ if and only if $|a| \le r$, one has  
\begin{equation*}
\iota(\zeta_{a,r}) \ = \ \left\{ \begin{array}{ll} 
                                \zeta_{0,1/r} & \text{if $|a| \le r$ \ , } \\
                                \zeta_{1/a,r/|a|^2} & \text{if $|a| > r$ \ . }
                                 \end{array} \right.
\end{equation*}

In particular, the action of  $\GL_2(K)$ on $\PP^1(K)$ though linear fractional transformations extends 
to an action on $\BPP^1_K$, which is transitive on type I points and type II points.   
The action of $\GL_2(K)$ preserves the logarithmic path distance:  
$\rho(\gamma(x),\gamma(y)) = \rho(x,y)$ for all $x, y \in \HH^1_K$ and $\gamma \in \GL_2(K)$.   
The stabilizer of $\zeta_G$ in $\GL_2(K)$ is $K^{\times} \cdot \GL_2(\cO)$. 

\smallskip
If we identify $\PP^1(K)$ with $K \cup \{\infty\}$ and make the usual conventions for arithmetic operations 
involving $\infty$, the spherical metric on $\PP^1(K)$ is given by 
\begin{equation*}
\|x,y\| \ = \ \left\{ \begin{array}{ll} |x-y|     & \text{if $|x|, |y| \le 1$ \ ,} \\
                                        |1/x-1/y| & \text{if $|x|, |y| > 1$ \ ,} \\
                                        1         & \text{otherwise \ .}            \end{array}  \right.
\end{equation*}
For $x,y \ne \infty$, one has  $\|x,y\| = \ |x-y|/\big(\max(1,|x|) \max(1,|y|)\big)$.  
For all $x,y \in \PP^1(K)$, and all $\gamma \in \GL_2(\cO)$, 
one has $\|\gamma(x),\gamma(y)\| = \|x,y\|$.

\smallskip
We will use two ``diameter'' functions on $\BPP^1_K$.  
For $x \in \BPP^1_K$, the diameter of $x$ with respect to the point $\infty \in \PP^1(K)$ is given by 
\begin{equation}
\diam_\infty(x) \ = \ \left\{ \begin{array}{ll} 
                                 0 & \text{if $x \in K$\ , } \\
                                 r & \text{if $x = \zeta_{a,r}$ is of type II or III \ ,} \\
                                 \inf\{ r : \zeta_{a,r} \in (x,\infty) \} & \text{if $x$ is of type IV \ ,} \\ 
                                 \infty & \text{if $x = \infty \in \PP^1(K)$\ . } 
                              \end{array} \right.  \label{diam_infty} 
\end{equation}
The function $\diam_\infty(x)$ is preserved by translations:  for any $b \in K$, if $\gamma(z) = z + b$,
then  $\diam_\infty(\gamma(x)) = \diam_\infty(x)$.
 
 The diameter with respect to the Gauss point $\zeta_G$ is defined by 
\begin{equation} \label{diam_G} 
\diam_G(x) \ = \ q^{-\rho(\zeta_G,x)} \ . 
\end{equation} 
If $x = \zeta_{a,r}$ is a point of type II or III, one has 
\begin{equation}
\diam_G(x) \ = \ \left\{ \begin{array}{ll} r     & \text{if $|a|, |r| \le 1$} \\
                                           r/|a|^2  & \text{if $|a| > 1$ and $r < |a|$} \\
                                           1/r      & \text{if $r > 1$ and $|a| \le r$} \end{array}  \right. 
                                                         \label{diam_G1}\\
\end{equation}
Evidently $0 \le \diam_G(x) \le 1$, with $\diam_G(x) < 1$ if $x \ne \zeta_G$,   
and $\diam_G(x) = 0$ if and only $x \in \PP^1(K)$.  
For each $\gamma \in \GL_2(\cO)$, one has $\diam_G(\gamma(x)) = \diam_G(x)$.
Moreover, $\GL_2(\cO)$ acts transitively on type I points, and, for a given $r \in |K^\times|$, 
on type II points with  $\diam_G(x) = r$.  If $a, b \in \PP^1(K)$, then   
\begin{equation*}  
\diam_G( a \vee_G b) \ = \ \|a,b\| \ . 
\end{equation*}

\smallskip
The Favre-Rivera Letelier metric $d(x,y)$ is defined by 
\begin{equation} \label{FRLdMetricDef}
d(x,y) \ = \ \big(\diam_G(x \vee_G y) - \diam_G(x)\big) + \big(\diam_G(x \vee_G y) - \diam_G(y)\big) \ .
\end{equation}
One has $0 \le d(x,y) \le 2$ for all $x, y$.
To see that $d(x,y)$ is a metric, note that it is positive if $x \ne y$, and is  
 clearly symmetric. It satisfies the triangle inequality $d(x,z) \le d(x,y) + d(y,z)$ because it is additive on paths:
if $Q$ is any point in $[x,y]$, then $d(x,y) = d(x,Q) + d(Q,y)$.    

\begin{proposition} \label{FRLd_Properties}
The metric $d(x,y)$ has the following properties:  

$(1)$ $d(\gamma(x),\gamma(y)) = d(x,y)$ for each $\gamma \in \GL_2(\cO);$    

$(2)$ $d(a,b) = 2 \|a,b\|$ for all $a,b \in \PP^1(K)$.
\end{proposition}

\begin{proof}  Assertion (1) follows from the definition of $d(x,y)$ and the fact that $K^\times \cdot \GL_2(\cO)$ 
stabilizes $\zeta_G$ and preserves the metric $\rho(x,y)$.  Assertion (2) follows from the fact that 
$\diam_G(a \vee_G b) = \|a,b\|$.
\end{proof} 

\smallskip
In addition to the balls $\cB_Q(\vv)^-$ and $\cB_\rho(x,r)^-$ introduced above,
we will use several other kinds of balls and discs.      
In naming them, we make the convention that {\em Roman letters} 
will be used for sets in $\AA^1(K)$ or $\PP^1(K)$,
and {\em script letters} for ones in $\BAA^1_K$ or $\BPP^1_K$.  
Also, we  speak of {\em discs} in $\AA^1(K)$ and $\BAA^1_K$, and {\em balls} in $\PP^1(K)$ and $\BPP^1_K$. 
 
For each $a \in K$ and $0 < r < \infty$ we have the classical discs   
\begin{equation*}
D(a,r)^- = \{z \in \AA^1(K) : |z-a| < r \}, \quad D(a,r) = \{z \in \AA^1(K) : |z-a| \le r \}. 
\end{equation*}
The associated Berkovich discs are 
\begin{equation*}
\cD(a,r)^- = \{x \in \BAA^1_K : \zeta_{a,r} \in (x,\infty] \}, \quad 
\cD(a,r)\  = \{x \in \BAA^1_K : \zeta_{a,r} \in [x,\infty] \} .
\end{equation*} 
Note that $\cD(a,r)^-$ is the path-component of $\BPP^1_K \backslash \{\zeta_{a,r}\}$ containing $D(a,r)^-$,
and $\cD(a,r)$ is the union of $\{\zeta_{a,r}\}$ and the path components of $\BPP^1_K \backslash \{\zeta_{a,r}\}$ 
which do not contain $\infty$.   If $\vv_a \in T_{\zeta_{a,r}}$ points  
towards $a$, and $\vv_\infty \in T_{\zeta_{a,r}}$ points towards $\infty$, then 
\begin{equation*}
\cD(a,r)^- \ = \ \cB_{\zeta_{a,r}}(\vv_a)^- \ , \quad \cD(a,r) = \BPP^1_K \backslash \cB_{Q_{a,r}}(\vv_\infty)^- \ .
\end{equation*} 
For either the weak or strong topology, 
$\cD(a,r)^-$ is open, and $\cD(a,r)$ is closed.

Given $a \in \BPP^1_K$ and $0 < r < 1$, we write
\begin{equation*}
B(a,r)^- = \{z \in \PP^1(K) : \|z,a\| < r \} \ , \quad B(a,r) = \{z \in \PP^1(K) : \|z,a\| \le r \}. 
\end{equation*}
There is a unique point $Q_{a,r} \in [a,\zeta_G]$ for which $\diam_G(Q_{a,r}) = r$.
The associated Berkovich balls are   
\begin{equation*}
\cB(a,r)^- = \{x \in \BPP^1_K : Q_{a,r} \in (x,\zeta_G] \} , \quad
\cB(a,r) =  \{z \in \BPP^1_K : Q_{a,r} \in [x,\zeta_G] \}. 
\end{equation*}
When $r = 1$, we define $B(a,1)^- = \bigcup_{r < 1} B(a,r)^-$ and $\cB(a,1)^- = \bigcup_{r < 1} \cB(a,r)^-$. 
Note that $\cB(a,r)^-$ is the path-component of $\BPP^1_K \backslash \{Q_{a,r}\}$ containing $B(a,r)^-$,
and $\cB(a,r)$ is the union of $\{Q_{a,r}\}$ and the path components of $\BPP^1_K \backslash \{Q_{a,r}\}$ 
which do not contain $\zeta_G$.  
For either the weak or strong topology, 
$\cB(a,r)^-$ is open, and $\cB(a,r)$ is closed.  If $\vv_a \in T_{Q_{a,r}}$ is the tangent vector pointing 
towards $a$, one has 
\begin{equation*}
\cB(a,r)^- \ = \ \cB_{Q_{a,r}}(\vv_a)^- \ , \quad \cB(a,r) = \BPP^1_K \backslash \cB_{Q_{a,r}}(\vv_{\zeta_G})^- \ .
\end{equation*} 
For any $\gamma \in \GL_2(\cO)$, one has $\gamma(\cB(a,r)^-) = \cB(\gamma(a),r)^-$ and 
$\gamma(\cB(a,r)) = \cB(\gamma(a),r)$.

\smallskip
Given a nonconstant function $\varphi(z) \in K(z)$, we define its Berkovich Lipschitz constant (relative to the 
Favre-Rivera-Letelier metric $d(x,y)$), to be 
\begin{equation} \label{LipschitzConstDef} 
\Lip_\Berk(\varphi) \ = \ \sup_{\substack{x, y \in \BPP^1_K \\ x \ne y}} \frac{d(\varphi(x), \varphi(y))}{d(x,y)} \ . 
\end{equation}  

\begin{proposition}  \label{LipInvariance} Let $\varphi(z) \in K(z)$ have degree $d \ge 1$.  
Then for any $\gamma_1, \gamma_2 \in \GL_2(\cO)$, one has $\Lip_\Berk(\gamma_1 \circ \varphi \circ \gamma_2) = \Lip_\Berk(\varphi)$.
\end{proposition} 

\begin{proof}  This follows from the definition of $\Lip_\Berk(\varphi)$
and the fact that $\GL_2(\cO)$ preserves $d(x,y)$. 
\end{proof}   

\section{Preliminary Lemmas} \label{PreliminaryLabelSection}

In this section we prove some lemmas which reduce bounding $\Lip_\Berk(\varphi)$ on $\BPP^1_K$ to bounding it 
on a restricted class of segments $[x,y]$.  

\begin{definition} \label{Radial-Limited}  Fix $0 < B_0 \le 1$.  
A segment $[b,c] \subset \BPP^1_K$ will be called {\em radial} if it is contained in a segment $[\xi,\zeta_G]$,
and it will be called {\em $B_0$-limited} if it is either contained in a segment $[\alpha,\xi]$ 
where  $\alpha \in \PP^1(K)$ and $\diam_G(\xi) = B_0$, 
or in a segment $[\xi,\zeta_G]$, where $\diam_G(\xi) = B_0$.
\end{definition}  

\begin{lemma} \label{DecompositionLemma}  Let $\varphi(z) \in K(z)$ have degree $d \ge 1$, put $B_0 = B_0(\varphi)$, 
and let $[b,c] \in \BPP^1_K$ be a segment.  Then there is a finite partition $\{a_1, \ldots, a_{n+1}\}$ of $[b,c]$ 
such that $a_1 = b$,  $a_{n+1} = c$, and each of $a_2, \ldots, a_n$ is of type {\rm II}, 
such that for each $i = 1, \ldots, n$, 

\begin{enumerate}
\item $\varphi$ maps the segment $[a_i,a_{i+1}]$ homeomorphically onto $[\varphi(a_i),\varphi(a_{i+1})];$

\item $[a_i,a_{i+1}]$ and $[\varphi(a_i),\varphi(a_{i+1})]$ are both radial, 
and $[a_i,a_{i+1}]$ is $B_0$-limited$;$ 

\item there is an integer $1 \le \delta_i \le d$ such that $\deg_\varphi(x) = \delta_i$ for each $x \in (a_i,a_{i+1});$ 

\item $\rho(\varphi(x),\varphi(y)) = \delta_i \cdot \rho(x,y)$ for all $x,y \in [a_i,a_{i+1}];$ and  

\item there are a constant $C_i > 0$ and an integer $k_i = \pm \delta_i$ such that for each $x \in [a_i,a_{i+1}]$,
if we put $r = \diam_G(x)$ and $R = \diam_G(\varphi(x))$, then $R = C_i \cdot r^{k_i}$. 
\end{enumerate}  
\end{lemma} 

\begin{proof} The existence of a partition $\{a_1, \ldots, a_{n+1}\}$ satisfying conditions $(1)$, $(3)$ and $(4)$ 
is due to Rivera-Letelier (see \cite{R-L2}, Corollaries 4.7 and 4.8, or \cite{B-R}, Theorem 9.33).  To refine the partition  
so that it satisfies (2), successively carry out the following adjunctions:
\begin{itemize}
\item[(A)]  To assure that each segment $[a_i,a_{i+1}]$ is radial, adjoin $\zeta_G$ to the partition if $\zeta_G \in [b,c]$, 
and for each $[a_i,a_{i+1}]$ which is now not radial, let $t_i = a_i \wedge_G a_{i+1}$ 
be the nearest point in $[a_i,a_{i+1}]$ to $\zeta_G$, and adjoint it to the partition. 
\item[(B)] To assure that each each segment $[a_i,a_{i+1}]$ is $B_0$-limited, 
for each $[a_i,a_{i+1}]$ which is not $B_0$-limited, let $\xi_i \in [a_i,a_{i+1}]$ be the unique point
with $\diam_G(\xi_i) = B_0$, and adjoin it to the partition;  
\item[(C)]  To assure that each segment $[\varphi(a_i),\varphi(a_{i+1})]$ is radial, 
consider each of the finitely many pre-images of $\zeta_G$ under $\varphi$, and if it belongs to $[b,c]$,
then adjoint it to the partition.  
\end{itemize} 
Assertion (5) is now immediate.  
There is a base $q > 1$ such that for each $x \in \HH^1_K$ one has $\rho(\zeta_G,x) = -\log_q(\diam_G(x))$.  
Hence $\rho(\zeta_G,x) = -\log_q(r)$ and $\rho(\zeta_G,\varphi(x)) = -\log_q(R)$.  By (2) and (4), 
for an appropriate choice of $k_i = \pm \delta_i$, for each  $x \in [a_i,a_{i+1}]$, 
\begin{equation*}
\rho(\zeta_G,\varphi(x))-\rho(\zeta_G,\varphi(a_i)) \ = \ k_i \cdot \big(\rho(\zeta_G,x)-\rho(\zeta_G,a_i)\big) 
\end{equation*}  
and (5) follows by exponentiating this.
\end{proof} 

\begin{corollary} \label{SupCor}  Let $\varphi(z) \in K(z)$ have degree $d \ge 1$, and put $B_0 = B_0(\varphi)$.
Let $\cI(B_0)$ be the collection of all radial segments of the form $[\alpha,\xi]$ or $[\xi,\zeta_G]$, 
where $\alpha \in \PP^1(K)$ and $\diam_G(\xi) = B_0$.  
If $L \in \RR$ is an upper bound for $\{\Lip_\Berk(\varphi\vert_I)  : I \in \cI(B_0)\}$, 
then $L$ is an upper bound for $\Lip_\Berk(\varphi)$.   
\end{corollary} 

\begin{proof}  We must show that $d(\varphi(b),\varphi(c)) \le L \cdot d(b,c)$ for all $b, c \in \BPP^1_K$.
This is trivial if $b = c$, so we can assume $b \ne c$.

First suppose $b$ and $c$ are of type II, and take 
a partition $\{a_1, \ldots, a_{n+1}\}$ 
of $[b,c]$ satisfying the conditions of Lemma \ref{DecompositionLemma}.
Each subsegment $[a_i,a_{i+1}]$ is contained in some $I \in \cI(B_0)$, 
so $\Lip_\Berk(\varphi\vert_{[a_i,a_{i+1}]}) \le \Lip_\Berk(\varphi\vert_I) \le L$. 
Furthermore $d(b,c) = \sum_{i=1}^n d(a_i,a_{i+1})$, so 
\begin{eqnarray*} 
d(\varphi(b),\varphi(c)) & \le & \sum_{i=1}^n d\big(\varphi(a_i),\varphi(a_{i+1})\big) \ \le \ 
\sum_{i=1}^n \Lip_\Berk(\varphi\vert_{[a_i,a_{i+1}]}) \cdot d(a_i,a_{i+1}) \\
                         & \le & L \cdot \sum_{i=1}^n d(a_i,a_{i+1}) \ = \ L \cdot d(b,c) \ . 
\end{eqnarray*}  

Now let $b \ne c$ in $\BPP^1_K$ be arbitrary.  Choose an exhaustion of $(b,c)$ by segments 
\begin{equation*}
[b^{(1)},c^{(1)}] \ \subset \ [b^{(2)},c^{(2)}] \ \subset \ \cdots \subset \ [b^{(j)},c^{(j)}]
\ \subset \ \cdots \subset \ (b,c) \ .
\end{equation*} 
with type II endpoints.  Then 
\begin{equation*}
d\big(\varphi(b),\varphi(c)\big) \ = \ \lim_{j \rightarrow \infty} d\big(\varphi(b^{(j)}),\varphi(c^{(j)})\big)
              \ \le \ \lim_{j \rightarrow \infty} L \cdot d\big(b^{(j)},c^{(j)}\big) \ = \ L \cdot d(b,c) \ .
\end{equation*}
Letting $b$ and $c$ range over $\BPP^1_K$, we see that $\Lip_\Berk(\varphi) \le L$.  
\end{proof} 

\begin{corollary} \label{DerivCor}  Let $\varphi(z) \in K(z)$ have degree $d \ge 1$, and put $B_0 = B_0(\varphi)$.
Let $I = [b,c]$ be a segment in $\BPP^1_K$ with $b \ne c$, and let $\{a_1, \ldots, a_{n+1}\}$ be a partition 
of $[b,c]$ with the properties in Lemma {\rm \ref{DecompositionLemma}}.  For each $i = 1, \ldots, n$, 
put $r_i = \min(\diam_G(a_i),\diam_G(a_{i+1}))$, $s_i = \max(\diam_G(a_i),\diam_G(a_{i+1}))$, 
and define $F_{\varphi,i} : [r_i,s_i] \rightarrow \RR$ by $F_{\varphi,i}(r) = C_i \cdot r^{k_i}$.  Then 
\begin{equation*} 
\Lip_\Berk(\varphi\vert_I) \ = \ \textstyle{\max_{1 \le i \le n} 
\Big(\sup_{r \in (r_i,s_i)} \big|F_{\varphi,i}^{\prime}(r)\big| \Big)} \ .
\end{equation*} 
\end{corollary} 

\begin{proof}  It suffices to show that for each $i$, 
$\Lip_\Berk(\varphi\vert_{[a_i,a_{i+1}]}) = \sup_{r \in (r_i,s_i)} |F_{\varphi,i}^{\prime}(r)|$.  

Take $x \ne y$ in $[a_i,a_{i+1}]$, and put $u = \diam_G(x)$, $v = \diam_G(y)$.  Without loss we can assume that $r < s$.  
By the Mean Value Theorem there is an $r_* \in (u,v)$ such that
\begin{equation} \label{MVT} 
F_{\varphi,i}(v) - F_{\varphi,i}(u) \ = \ F_{\varphi,i}^{\prime}(r_*) \cdot (v-u) \ ,
\end{equation} 
so $d(\varphi(x),\varphi(y)) = |F_{\varphi,i}^{\prime}(r_*)| \cdot d(x,y)$.  Hence 
$\Lip_\Berk(\varphi\vert_{[a_i,a_{i+1}]}) \le \sup_{r \in (r_i,s_i)} |F_{\varphi,i}^{\prime}(r)|$. 

The opposite inequality 
follows from the fact that $F_{\varphi,i}^{\prime}(r)$ is continuous:  for each $r_\# \in (r_i,s_i)$ and each $\varepsilon > 0$,
there is a $\delta > 0$ such that $[r_\#-\delta,r_\#+\delta] \subset (r_i,s_i)$, and 
$|F_{\varphi,i}^\prime(t) - F_{\varphi,i}^\prime(r_\#)| < \varepsilon$ for all $t \in [r_\#-\delta,r_\#+\delta]$.
Take $x, y \in [a_i,a_{i+1}]$ with $r_\#-\delta < \diam_G(x) < \diam_G(y) < r_\#+\delta$, and let $r_*$ be as in (\ref{MVT}) 
for this choice of $x, y$.  Then $|F_{\varphi,i}^\prime(r_*) - F_{\varphi,i}^\prime(r_\#)| < \varepsilon$.  
It follows that 
$\Lip_\Berk(\varphi\vert_{[a_i,a_{i+1}]}) \ge \sup_{r \in (r_i,s_i)} |F_{\varphi,i}^{\prime}(r)|$.  
\end{proof}

\begin{lemma} \label{fprimeMonoLemma}
Let $\Phi(z) \in K(z)$ have degree $d \ge 1$; write $B_0 = B_0(\Phi)$, put $c_0 = \Phi(0)$, and assume $\Phi(D(0,B_0)^-) = D(c_0,R)^-$.  Expand 
\begin{equation*}
\Phi(z) \ = \ c_0 + (c_1 z + c_2 z^2 + \cdots c_n z^n) \cdot U(z) \ , 
\end{equation*}
on $D(0,B_0)^-$, where $U(z)$ is a unit power series.      
Then we can partition $[0,B_0]$ into finitely many subintervals $[r_i,r_{i+1}]$,
where $0 = r_1 < \cdots < r_{\ell+1} = B_0$, such that on $[r_i,r_{i+1}]$ we have  
\begin{equation*}
f_\Phi(r) \ = \ f_i(r) \ := \ |c_{k(i)}| \cdot r^{k(i)} 
\end{equation*} 
for a suitable index $k(i)$.  Let $f_\Phi^\prime(r) = \lim_{h \rightarrow 0^+} (f_\Phi(r+h)-f_\Phi(r))/h$ 
be the right-derivative of $f$ on $[0,B_0)$.  Then $f_\Phi^\prime(r)$ is non-decreasing on $[0,B_0)$,
and for each $i = 1, \ldots, \ell-1$ we have $k(i) \le k(i+1)$.   
\end{lemma} 

\noindent{\bf Remark.} There is a minimal partition with the properties in Lemma \ref{fprimeMonoLemma}, 
which has the additional property that $k(i) < k(i+1)$ for $i = 1, \ldots, \ell-1$.  However, in the applications
the partition we use may not be minimal, so we only assume that  $k(i) \le k(i+1)$.  

\begin{proof}
We will regard each $f_i(r) = \ |c_{k(i)}| \cdot r^{k(i)}$ as defined for all $r \ge 0$.  
Clearly $f_i(r)$ is continuous and monotone increasing, 
and $f_i^\prime(r) = k(i) |c_{k(i)}| r^{k(i)-1}$ is continuous and nondecreasing. 
   
Since $f_\Phi(r)$ is continuous and monotone increasing,  
at each break point $r_i$ we must have $f_{i-1}^\prime(r_i) \le f_i^\prime(r_i)$.  Thus $f_\Phi^{\prime}(r)$ is non-decreasing.
Furthermore at each such $r_i$
\begin{equation*} 
k(i-1) \cdot \frac{f_{i-1}(r_i)}{r_i} \ = \ f_{i-1}^\prime(r_i) 
\ \le \ f_i^{\prime}(r_i) \ = \ k(i) \cdot \frac{f_i(r_i)}{r_i} \ ,
\end{equation*}
so since $f_{i-1}(r_i) = f_i(r_i)$ we must have $k(i-1) \le k(i)$.
\end{proof} 

\section{Lipschitz constants for Linear Fractional Transformations} \label{LinearBoundSection}

When $\varphi \in \PGL_2(K)$, one can find its Lipschitz constants exactly:

\begin{theorem} \label{LinearThm} Let $K$ be a complete, algebraically closed nonarchimedean field,  
and let $\varphi(z) = (az+b)/(cz+d) \in K(z)$ have degree $d =1$. Then $B_0(\varphi) = 1$, and 
\begin{eqnarray*}
\Lip_\Berk(\varphi) & = & \Lip_{\PP^1(K)}(\varphi) \ = \ \frac{1}{\GIR(\varphi)} \ = \ \frac{1}{\GPR(\varphi)} \\
                    & = & \frac{1}{|\Res(\varphi)|} \ = \ \frac{\max(\,|a|,|b|,|c|,|d|\,)}{|ad-bc|} \ .
\end{eqnarray*}
\end{theorem} 

\begin{proof} Theorem \ref{LinearThm} is a restatement of Theorem \ref{MobiusCase} in the Introduction.
Write $[\varphi]$ for the matrix 
\begin{equation*}
\left[\begin{array}{cc} a & b \\ c & d \end{array} \right] \ . 
\end{equation*}
Since $\varphi$ is unchanged when $[\varphi]$ is scaled by an element of $K^{\times}$, 
we can assume that $[\varphi] \in M_2(\cO)$ and $\max(|a|,|b|,|c|,|d|) = 1$.  
Since $d(x,y)$ and $\|x,y\|$ are preservied by $\GL_2(\cO)$,
we can pre- and post-compose $\varphi$ with elements of $\GL_2(\cO)$ 
without changing $\Lip_\Berk(\varphi)$ and $\Lip_{\PP^1(K)}(\varphi)$;  
such compositions also preserve $\GIR(\varphi)$, $\GPR(\varphi)$, 
the value of $|\Res(\varphi)| = |ad-bc|$, 
and the fact that $\max(|a|,|b|,|c|,|d|) = 1$.
Choosing $\gamma_1, \gamma_2 \in \GL_2(\cO)$ to
carry out appropriate combinations of elementary row and column operations, 
and setting $\Phi  = \gamma_1 \circ \varphi \circ \gamma_2$, we can arrange that  
\begin{equation*}
[\Phi] \ = \ \left[\begin{array}{cc} 1 & 0 \\ 0 & D \end{array} \right]
\end{equation*} 
where $D \in \cO \backslash \{0\}$.  
Note that  $\Phi(\zeta_G) = \zeta_{0,1/|D|}$, so $\GIR(\varphi) = \GIR(\Phi) = |D| = |\Res(\varphi)|$.  
Similarly $\Phi(\zeta_{0,|D|}) = \zeta_G$, so $\GPR(\varphi) = \GPR(\Phi) = |D|$.  
Trivially $B_0(\varphi) = B_0(\Phi) = 1$.  

We will now show that $\Lip_\Berk(\Phi) = 1/|D|$. Recall that  
\begin{equation*}
\diam_G(\zeta_{a,r}) \ = \ \left\{ \begin{array}{ll} r & \text{ if $|a|, r \le 1$ \ ,} \\
													r/|a|^2 & \text{ if $|a| > 1$ and $r < |a|$ \ ,} \\
													1/r & \text{ if $|a| \le 1$ and $r \ge 1$, or if $1 < |a| \le r$ \ .} 
													\end{array} \right.
\end{equation*} 
For future use, note that the three formulas on the right can be combined the a single expression  
\begin{equation} \label{diamGformula} 
\diam_G(\zeta_{a,r}) \ = \ \frac{r}{\max(1,|a|,r)^2} \ .
\end{equation} 
Given a point $\zeta_{a,r} \in \AA^1_\Berk$, we have $\Phi(\zeta_{a,r}) = \zeta_{a/D,r/|D|}$.   
It follows that 
\begin{equation*}
\diam_G(\Phi(\zeta_{a,r})) \ = \ \left\{ \begin{array}{ll} r/|D| & \text{ if $|a|, r \le |D|$ \ ,} \\
									r |D| / |a|^2 & \text{ if $|a| > |D|$ and $r < |a|$ \ ,} \\
									|D|/r & \text{ if $|a| \le |D|$ and $r \ge |D|$, or if $|D| < |a| \le r$ \ .} \\
													 \end{array} \right.
\end{equation*} 
 
Since points of type II and III are dense in $\BPP^1_K$ for the strong topology, 
it suffices to bound $\Lip_\Berk(\varphi)$ on paths $[a,\infty]$ where $a \in K$.  The remainder of 
the argument is a case by case verification.

Fix $a \in K$, and consider a point $\zeta_{a,r}$. If $|a| \le |D|$ and $s < r \le |D|$ then 
\begin{equation*}
\frac{d(\Phi(\zeta_{a,r}),\Phi(\zeta_{a,s}))}{d(\zeta_{a,r},\zeta_{a,s})}
 \ = \ \frac{r/|D| - s/|D|}{r-s} \ = \ \frac{1}{|D|} \ . 
\end{equation*}
If $|a| \le |D|$ and $|D| \le s < r$, then $|a/D| \le 1$ while $ 1 \le s/|D| < r/|D|$, so 
\begin{equation*}
\frac{d(\Phi(\zeta_{a,r}),\Phi(\zeta_{a,s}))}{d(\zeta_{a,r},\zeta_{a,s})}
 \ = \ \frac{|D|/s - |D|/r}{r-s} \ = \ \frac{|D|}{rs} \ < \ \frac{1}{|D|}\ . 
\end{equation*}
If $|D| < |a| \le 1$ and $0 < s < r \le |a|$, then  
\begin{equation*}
\frac{d(\Phi(\zeta_{a,r}),\Phi(\zeta_{a,s}))}{d(\zeta_{a,r},\zeta_{a,s})}
\ = \ \frac{(r/|D|)/(|a/D|)^2 - (s/|D|)/(|a/D|)^2)}{r-s} \ = \ \frac{|D|}{|a|^2} \ < \ \frac{1}{|D|}\ . 
\end{equation*}
If $|D| < |a| \le 1$ and $|a| \le s < r$, then $1 < |a/D| \le s/|D| < r/|D|$ so again 
\begin{equation*}
\frac{d(\Phi(\zeta_{a,r}),\Phi(\zeta_{a,s}))}{d(\zeta_{a,r},\zeta_{a,s})}
 \ = \ \frac{|D|/s - |D|/r}{r-s} \ = \ \frac{|D|}{rs}  \ < \ \frac{1}{|D|}\ . 
\end{equation*}
If $|a| > 1$ and $0 < s < r \le |a|$, then  
\begin{equation*}
\frac{d(\Phi(\zeta_{a,r}),\Phi(\zeta_{a,s}))}{d(\zeta_{a,r},\zeta_{a,s})}
 \ = \ \frac{(r/|D|)/(|a|/|D|)^2 - (s/|D|)/(|a|/|D|)^2)}{r/|a|^2-s/|a|^2} \ = \ |D| \ \le \ \frac{1}{|D|}\ . 
\end{equation*}
Finally, if $|a| > 1$ and $|a| \le s < r$, then $\zeta_{a,s} = \zeta_{0,s}$ and $\zeta_{a,r} = \zeta_{0,r}$ so 
\begin{equation*}
\frac{d(\Phi(\zeta_{a,r}),\Phi(\zeta_{a,s}))}{d(\zeta_{a,r},\zeta_{a,s})}
 \ = \ \frac{|D|/s - |D|/r}{1/s-1/r} \ = \ |D| \ \le \ \frac{1}{|D|}\ . 
\end{equation*}
Thus $\Lip_\Berk(\varphi) = \Lip_\Berk(\Phi) = 1/|D|$.

Clearly $\Lip_{\PP^1(K)}(\Phi) \le \Lip_\Berk(\Phi) = 1/|D|$.  
To prove that $\Lip_{\PP^1(K)}(\Phi) = 1/|D|$, it suffices to show that $\Lip_{\PP^1(K)}(\Phi) \ge 1/|D|$. 
This is trivial, since if $x = 0$ and $y = D$, then  $\|x,y\|=\|0,D\|=|D|$ and $\|\Phi(x),\Phi(y)\|=\|0,1\|=1$.  
\end{proof} 

\section{Some Auxiliary Constants} \label{AuxSection}

In this section, we study the four constants associated to $\varphi$ in the Introduction: the Gauss Pre-Image radius, 
the Root-Pole number, the Ball-Mapping radius, and the Gauss Image radius.

\begin{definition} \label{FourDefs}
Let $\varphi(z) \in K(z)$ have degree $d \ge 1$.   

$(A)$ The Gauss Image radius of $\varphi$ is $\GIR(\varphi) = \diam_G(\varphi(\zeta_G))$.
\end{definition} 

$(B)$ The Root-Pole number of $\varphi$ is
\begin{equation*}
\RP(\varphi) \ = \ \min \{\|\alpha, \beta\| : \alpha, \beta \in \PP^1(K), \varphi(\alpha) = 0, \varphi(\beta) = \infty \} \ .
\end{equation*} 

$(C)$ The Ball-Mapping radius of $\varphi$ is  
\begin{equation*}
B_0(\varphi) = \sup \{0 < r \le 1 : \text{for all $a \in \PP^1(K)$, $\varphi(\cB(a,r)^-) \ne \BPP^1_K$}\} \ .
\end{equation*}

$(D)$ The Gauss Pre-Image radius of $\varphi$ is  
\begin{equation*}
\GPR(\varphi) \ = \ \min \{ \diam_G(x) : x \in \BPP^1_K, \varphi(x) = \zeta_G \} \ .  
\end{equation*}

Clearly the Ball-Mapping radius, the Gauss Image radius, and the Gauss Pre-Image radius are invariant under pre- and post- 
composition of $\varphi$ with elements of $\GL_2(\cO)$; the Ball-Mapping radius is also invariant under post-composition 
of $\varphi$ with elements of $\GL_2(K)$. 
The Root-Pole number is not invariant under either pre- or post- 
composition by $\GL_2(\cO)$, but it lies between the Gauss Pre-Image radius and the Ball-Mapping radius:  

\begin{proposition} \label{GaussInequalitiesProp} Let $\varphi(z) \in K(z)$ have degree $d \ge 1$.  Then 
\begin{equation*}
0 \ < \ \GPR(\varphi) \ \le \ \RP(\varphi) \ \le \ B_0(\varphi) \ \le \ 1 \ .
\end{equation*}
\end{proposition}

\begin{proof} 
Since there are most $d$ pre-images of $\zeta_G$ under $\varphi$, which all lie in $\HH^1_K$, clearly $\GPR(\varphi) > 0$.
Also, by the definition of $B_0(\varphi)$, we trivially have $B_0(\varphi) \le 1$.

It is also easy to see that $\GPR(\varphi) \le \RP(\varphi)$.  Indeed, if $r = \RP(\varphi)$, then there are a root $\alpha$
and a pole $\beta$ of $\varphi$ with $\|\alpha,\beta\| = r$.  The image of the path $[\alpha,\beta]$ under $\varphi$ 
is connected and contains $0$ and $\infty$, so it contains the path $[0,\infty]$.  Hence it contains $\zeta_G$, 
and there is a point $x$ of $\varphi^{-1}(\zeta_G)$ in $[\alpha,\beta]$.  It follows that
$r = \diam_G(\alpha \wedge_G \beta) \ge \diam_G(x) \ge \GPR(\varphi)$.   

Finally, we show that $\RP(\varphi) \le B_0(\varphi)$.  If $B_0(\varphi) = 1$, then trivially $\RP(\varphi) \le B_0(\varphi)$,
since $\|\alpha,\beta\| \le 1$ for any pair of elements $\alpha, \beta \in \PP^1(K)$.  Suppose $B_0(\varphi) < 1$, 
and take any $r$ with $B_0(\varphi) < r \le 1$. Since $r > B_0(\varphi)$, there is a ball $\cB(a,r)^-$ 
with $\varphi(\cB(a,r)^-) = \BPP^1_K$.  Hence there are $\alpha, \beta \in \PP^1(K) \cap \cB(a,r)^-$ such that 
$\varphi(\alpha) = 0$, $\varphi(\beta) = \infty$.  It follows that $r > \|\alpha,\beta\| \ge \RP(\varphi)$.  Since 
$B_0(\varphi)$ is the infimum of all such $r$, we must have $B_0(\varphi) \ge \RP(\varphi)$. 
\end{proof} 

The inequalities $\GPR(\varphi) \le \RP(\varphi) \le B_0(\varphi)$ in Proposition 
\ref{GaussInequalitiesProp} can both be strict.  For example, consider the polynomial $\varphi(z) = z^2 - 1/p^2 \in \CC_p[z]$, 
where $p$ is an odd prime.  One sees easily that $\varphi^{-1}(\zeta_G) = \{ \zeta_{1/p,1/p}, \zeta_{-1/p,1/p} \}$ so 
$\GPR(\varphi) = p^{-3}$.  The zeros of $\varphi$ are $\{\pm 1/p\}$ and the only pole is $\{\infty\}$,
so $\RP(\varphi) = p^{-1}$.  Finally, the only solution 
to $\varphi(z) = -1/p^2$ is $z = 0$. It follows that if 
$\varphi(\cB(a,r)^-) = \BPP^1_K$ for some ball, then both $0, \infty \in \cB(a,r)^-$. 
This is impossible with $r < 1$, so $B_0(\varphi) = 1$.  

\medskip
Our next proposition says that $B_0(\varphi) \in |K^\times|$, and there is a ball which realizes it.

\begin{proposition}  \label{B_0RationalityProp}
Let $\varphi(z) \in K(z)$ have degree $d \ge 1$, and put $B_0 = B_0(\varphi)$.  Then $B_0 \in |K^{\times}|$. Moreover,
if $B_0 < 1$ $($so necessarily $d \ge 2)$, there is an $\alpha \in \PP^1(K)$ for which $\varphi(\cB(a,B_0)^-)$ is a ball, 
but $\varphi(\cB(\alpha,B_0)) = \BPP^1_K$.   
\end{proposition} 

\begin{proof} The proof uses the theory of the ``crucial set'' from (\cite{RR-GMRL}, \cite{RR-NEND}).

Write $B_0 = B_0(\varphi)$.  If $d=1$, then $B_0 = 1$, and the assertions are trivial. Assume $d \ge 2$.
If $B_0 = 1$, the assertions are again trivial, so we can assume $0 < B_0 < 1$.    
Choose a sequence of numbers $1 > R_1 > R_2 > \cdots > B_0$ in $|K^{\times}|$ with $\lim_{i \rightarrow \infty} R_i = B_0$, 
and a sequence of balls $\cB(a_i,R_i)^-$ such that $\varphi(\cB(a_i,R_i)^-) = \BPP^1_K$ for each $i$.  Each of the balls 
$\cB(a_i,R_i)^-$ can be written in the form $\cB_{P_i}(\vv_i)^-$ where $P_i$ is a type II point and $\vv_i \in T_{P_i}$ 
is a suitable tangent vector.
 
By the proof of Theorem 4.6 of \cite{RR-NEND} (alternately see Theorem 4.6 of \cite{RR-GMRL}), 
each $\cB(a_i,R_i)^-$ contains either a classical fixed point of $\varphi$ (that is, a fixed point in $\PP^1(K)$), or 
a repelling fixed point of $\varphi$ in $\BHH^1_K$ of a special type, a {\em focused repelling fixed point}. 
A focused repelling fixed point is a type II point $Q$ with $\varphi(Q) = Q$, such that $\deg_\varphi(Q) \ge 2$   
and there is a unique $\vv_\# \in T_Q$ for which $\varphi_*(\vv_\#) = \vv_\#$.  
(We are using the case of (\cite{RR-GMRL}, \cite{RR-NEND}, Theorem 4.6) concerning a ball with a type II boundary point:   
each such ball is dealt with by one of Lemmas 2.1, 2.2, and 4.5 of (\cite{RR-GMRL}, \cite{RR-NEND}).  
Lemmas 2.1 and 2.2 produce classical fixed points in $\cB(a_i,R_i)^-$, 
while Lemma 4.5 produces either a classical fixed point or a focused repelling fixed point.)
By (\cite{RR-GMRL}, \cite{RR-NEND}, Proposition 3.1) $\cB_Q(\vv_\#)^-$ contains all the classical fixed points of $\varphi$, 
and $\varphi(\BPP^1_K \backslash \cB_Q(\vv_\#)^-) = \BPP^1_K$.  There are at most $d+1$ classical fixed points of $\varphi$,
and by (\cite{RR-GMRL}, \cite{RR-NEND}, Corollary 6.3) 
there are at most $d-1$ repelling fixed points of $\varphi$ in $\BHH^1_K$.  
Thus we can apply the Pigeon-hole Principle to the balls and fixed points. 

First suppose there is a classical fixed point $\alpha$ which is contained in infinitely many balls $\cB(a_i,R_i)^-$.
By replacing the sequence of balls with a subsequence, we can assume $\alpha \in \cB(a_i,R_i)^-$ for each $i$.    
After conjugating $\varphi$ by a suitable element of $\GL_2(\cO)$, we can assume that $\alpha = 0$, and 
that the sequence of balls is $\{\cB(0,R_i)^-\}_{i \ge 1}$.  Suppose $B_0 \notin |K^{\times}|$.   
Then the point $P = \zeta_{0,B_0}$ is of type III.  The tangent space $T_P$ consists of two directions
$\vv_0,\vv_\infty$, and $\cB(0,B_0)^- = \cB_P(\vv_0)^-$.  Put $Q = \varphi(P)$, $\vw_1 = \varphi_*(\vv_0)$ 
and $\vw_2 = \varphi_*(\vv_\infty)$.  Necessarily $Q$ is of type III, and $\vw_1, \vw_2$ are the two tangent directions
in $T_Q$ (see \cite{B-R}, Corollary 9.20).  By the definition of the ball mapping radius, 
$\varphi(\cB_P(\vv_0)^-)$ is a ball, hence necessarily $\varphi(\cB_P(\vv_0))^-) = \cB_Q(\vw_1)^-$.  
By Rivera-Letelier's Annulus Mapping Theorem (see \cite{B-R}, Lemma 9.45), 
there is a point $P_1 \in \cB_P(\vv_\infty)^-$ for which $\varphi(\Ann(P,P_1))$ is 
an annulus $\Ann(Q,Q_1) \subset \cB_Q(\vw_2)^-$.  Without loss we can suppose  $P_1 = \zeta_{0,R}$ for some
$R > B_0$.  Since $\cB(0,R)^- = \cB(0,B_0)^- \cup \{P\} \cup \Ann(P,P_1)$, it follows that 
\begin{equation*}
\varphi(\cB(0,R)^-) \ = \ \cB_Q(\vw_1)^- \cup \{Q\} \cup \Ann(Q,Q_1) \ \ne \ \BPP^1_K \ .       
\end{equation*} 
This contradicts that $\varphi(\cB(0,R_i)^-) = \BPP^1_K$ when $B_0 < R_i < R$, hence $B_0 \in |K^{\times}|$.

By definition $\varphi(\cB(0,B_0)^-)$ is a ball;  we claim that $\varphi(\cB(0,B_0)) = \BPP^1_K$.  
Suppose this were not the case;
write $P = \zeta_{0,B_0}$ and put $Q = \varphi(P)$.  Let $\vv_\infty \in T_P$ be the direction containing $\infty$,
and put $\vw_\infty = \varphi_*(\vv_\infty) \in T_Q$.  The map 
$\varphi_* : T_P \rightarrow T_Q$ is surjective, so for each $\vw \in T_Q$ with $\vw \ne \vw_\infty$ there is 
some $\vv \in T_P$ with $\varphi_*(\vv) = \vw$.  
Since $\varphi(\cB_P(\vv)^-)$ contains $\cB_Q(\vw)^-$, we see that 
\begin{equation*}
\varphi(\cB(0,B_0)) 
\ \supseteq \ \{Q\} \cup \bigcup_{\vw \ne \vw_\infty} \cB_Q(\vw)^- \ \supseteq \ \BPP^1_K \backslash \cB_Q(\vw_\infty)^- \ . 
\end{equation*} 
Moreover, for each $\vv \in T_P$, the image $\varphi(\cB_P(\vv)^-)$ is either a ball or all of $\BPP^1_K$.
(If there were some $\vv \in T_P$ with $\vv \ne \vv_\infty$ 
for which $\varphi_*(\vv) = \vw_\infty$, then $\varphi(\cB(0,B_0))$ would contain $\cB_Q(\vw_\infty)^-$, 
hence would be $\BPP^1_K$.) 
Thus $\vv_\infty$ is the only direction in $T_P$ with $\varphi_*(\vv) = \vv_\infty$. 
It follows for each $\vv \ne \vv_\infty$, the image $\varphi(\cB_P(\vv)^-)$ is a ball $\cB_Q(\vw)^-$ 
with $\vw \ne \vv_\infty$, and  
that $\varphi(\cB(0,B_0)) = \BPP^1_K \backslash \cB_Q(\vw_\infty)^-$.  However, now Rivera-Letelier's Annulus 
mapping theorem shows there is a point $P_1 = \zeta_{0,S_1} \in \cB_P(\vv_\infty)^-$ for which $\varphi_*(\Ann(P,P_1))$
is the annulus $\Ann(Q,\varphi(P_1)) \subset \cB_Q(\vw_\infty)^-$.  This would mean that for each $R$ with $B_0 < R < S_1$
$\varphi(\cB(0,R)^-) \ne \BPP^1_K$, which contradicts the fact that $\varphi(\cB(0,R_i)^-) = \BPP^1_K$ for all $i$.  
Hence it must be that $\varphi(\cB(0,B_0)) = \BPP^1_K$.

Next consider the case where no classical fixed point is contained in infinitely many $\cB(a_i,R_i)^-$.  
In this situation there must be a focused
repelling fixed point $\xi$ which belongs to infinitely many $\cB(a_i,R_i)^-$.  After passing to a subsequence of the
balls, if necessary, we can assume that $\xi \in \cB(a_i,R_i)^-$ for each $i$, and that no classical fixed point is contained
in any $\cB(a_i,R_i)^-$.  After conjugating $\varphi$ by a suitable element of $\GL_2(\cO)$ if necessary,  
we can assume that $\xi = \zeta_{0,S_1}$ for some $S_1 \le B_0$, 
and that the sequence of balls is $\{\cB(0,R_i)^-\}_{i \ge 1}$. 
Since $\xi$ is of type II, necessarily $S_1 \in |K^{\times}|$.  
Since no $\cB(0,R_i)^-$ contains classical fixed points of $\varphi$,
the distinguished direction $\vv_\# \in T_\xi$ must be $\vv_\# = \vv_\infty$.  
It follows that $\BPP^1_K \backslash \cB_\xi(\vv_\#)^- = \cB(0,S_1)$, and $\varphi(\cB(0,S_1)) = \BPP^1_K$. If $B_0 > S_1$,
then $\varphi(\cB(0,R)^-) = \BPP^1_K$ for each $R$ with $B_0 > R > S_1$.  This contradicts the definition of the ball
mapping radius, so $B_0 = S_1 \in |K^{\times}|$.  The equality $B_0 = S_1$ also shows that $\varphi(\cB(0,B_0)^-)$ is a ball, 
but $\varphi(\cB(0,B_0)) = \BPP^1_K$.
\end{proof} 

There is a simple formula for $\GIR(\varphi)$ in terms of the coefficients of a normalized representation of $\varphi$:

\begin{proposition}\label{GIRProp}  Let $\varphi(z) \in K(z)$ have degree $d \ge 1$, 
and let $(F,G)$ be a normalized representation of $\varphi$.  
Write $F(X,Y) = a_d X^d + \ldots + a_1 X Y^{d-1} + a_0 Y^d$, $G(X,Y) = b_d X^d + \ldots + b_1 X Y^{d-1} + b_0 Y^d$.
Then 
\begin{equation} \label{GIRformula} 
\GIR(\varphi) \ = \ \max_{i \ne j} \Big( 
              \Big| \det \Big[ \begin{array}{cc} a_i & a_j \\ b_i & b_j \end{array} \Big] \Big| \Big) \ .
\end{equation}
\end{proposition}

\begin{proof} Suppose $\varphi(\zeta_G) = Q$.  After replacing $\varphi$ with $\gamma \circ \varphi$ for a suitable 
$\gamma \in \GL_2(\cO)$ we can assume that $Q = \zeta_{0,R}$, where $R = \GIR(\varphi)$.  Since $\gamma$ preserves 
$\diam_G(\cdot)$ and 
\begin{equation*} 
 \Big| \det \Big( \gamma \circ \Big[ \begin{array}{cc} a_i & a_j \\ b_i & b_j \end{array} \Big] \Big) \Big| \ = \ 
 \Big| \det \Big[ \begin{array}{cc} a_i & a_j \\ b_i & b_j \end{array} \Big] \Big|  \ ,
\end{equation*}
this does not affect (\ref{GIRformula}).  Since $(F,G)$ is normalized, for generic $z \in \cO_K$ we must have $|F(z,1)| = R$ and $|G(z,1)| = 1$.  This means that all coefficients of $F$ must satisfy $|a_i| \le R$ and at least one coefficient of $G$
must satisfy $|b_j| = 1$.

Next take $C \in K$ with $|C|=R$, and put $\Phi(z) = (1/C) \varphi(z)$.  Then $\Phi(\zeta_G) = \zeta_G$, 
so $\GIR(\Phi) = 1$,  and $(F_0(X,Y),G_0(X,Y)) := ((1/C)F(X,Y),G(X,Y))$ is a normalized representation of $\Phi$. 
Writing $F_0(X,Y) = A_d X^d + \ldots + A_1 X Y^{d-1} + A_0 Y^d$, $G_0(X,Y) = B_d X^d + \ldots + B_1 X Y^{d-1} + B_0 Y^d$,
it suffices to show that 
\begin{equation*} 
\max_{i \ne j} \Big( 
              \Big| \det \Big[ \begin{array}{cc} A_i & A_j \\ B_i & B_j \end{array} \Big] \Big| \Big) \ = \ 1 \ .
\end{equation*}
If this were not the case, then $\tA_i \tB_j - \tA_j \tB_i = \widetilde{0} \pmod{\fM}$ for all $i \ne j$, 
so one of the vectors $\tA = (\tA_d, \ldots, \tA_0)$, $\tB = (\tB_d, \ldots, \tB_0)$  
gotten by reducing the coefficients of $F_0, G_0 \pmod{\fM}$ would be a multiple of the other.  Hence 
$\Phi$ would have constant reduction at $\zeta_G$.  However, this contradicts (\cite{B-R}, Lemma 2.17), 
which says that $\Phi(\zeta_G) = \zeta_G$ if and only if $\Phi$ has nonconstant reduction.
\end{proof} 

\medskip
We next seek lower bounds for $B_0(\varphi)$, $\GIR(\varphi)$, and $\GPR(\varphi)$ in terms of  $|\Res(\varphi)|$.
For this, we will need the following proposition, which is a projective version of the classical formula 
for the resultant of two polynomials.

Let the zeros $\alpha_1, \ldots, \alpha_d$ of $\varphi$ in $\PP^1(K)$ (listed with multiplicity) 
have homogeneous coordinates 
\begin{equation*} 
(1:\sigma_1), \ldots, (1:\sigma_m), (\delta_{m+1}:1), \ldots, (\delta_d:1) \ , 
\end{equation*}
where $|\sigma_1|, \ldots, |\sigma_m| \le 1$ and $|\delta_{m+1}|, \ldots, |\delta_d| < 1$.    
Likewise, let the poles $\beta_1, \ldots, \beta_d$ of $\varphi$ (listed with multiplicity) 
have homogeneous coordinates   
\begin{equation*} 
(1:\tau_1), \ldots, (1:\tau_n), (\eta_{n+1}:1), \ldots, (\eta_d:1) \ , 
\end{equation*}
where $|\tau_1|, \ldots, |\tau_n| \le 1$ and $|\eta_{n+1}|, \ldots, |\eta_d| < 1$. 
  
Let $(F,G)$ be a normalized representation of $\varphi$. Then we can write 
\begin{eqnarray*}
F(X,Y) & = & C_0 \cdot \prod_{i=1}^m (X-\sigma_i Y) \cdot \prod_{i=m+1}^d (\delta_i X - Y) \ , \\
G(X,Y) & = & C_1 \cdot \prod_{j=1}^n (X-\tau_i Y) \cdot \prod_{j=n+1}^d (\eta_i X - Y) \ ,
\end{eqnarray*} 
where $C_0, C_1 \in K^{\times}$ satisfy $0 < |C_0|, |C_1| \le 1$ and $\max(|C_0|, |C_1|) = 1$.  

\begin{proposition} \label{ResultantProp} Let $\varphi \in K(z)$ have degree $d \ge 1$.  
With notations as above, we have 
\begin{equation*}
|\Res(\varphi)| \ = \ |C_0|^d |C_1|^d \cdot \prod_{i, j = 1}^d \|\alpha_i, \beta_j\| \ . 
\end{equation*} 
\end{proposition} 

\begin{proof}
By perturbing $\varphi$ slightly we can assume that none of its zeros or poles are the point $\infty = (0:1)$, 
while preserving the distances $\|\alpha_i,\beta_j\|$ and the absolute values $|C_0|$, $|C_1|$. 
(For instance, we can replace $\varphi$ with $\varphi \circ \gamma$ for 
a suitable $\gamma \in \GL_2(\cO)$, sufficiently close to the identity).  
If we expand 
\begin{eqnarray*} 
F(X,Y) & = & a_d X^d + a_{d-1} X^{d-1} Y + \cdots + a_0 Y^d \ , \\
G(X,Y) & = & b_d X^d + b_{d-1} X^{d-1} Y + \cdots + b_0 Y^d \ ,
\end{eqnarray*} 
then $|\Res(\varphi)| = |\Res(F,G)|$ where 
\begin{equation} \label{ResultantFormula}
\Res(F,G) \ = \ \det\Bigg( \ \left[ \begin{array}{cccccccc} 
                          a_d & a_{d-1} & \cdots &  a_1    &    a_0     &         &      &      \\
                              & a_d & a_{d-1}    & \cdots  &    a_1     &    a_0  &      &      \\
                              &     &            &         & \vdots     &         &      &      \\
                              &     &            &    a_d  &   a_{d-1}  & \cdots  & a_1  & a_0  \\
                          b_d & b_{d-1} & \cdots &  b_1    &    b_0     &         &      &      \\
                              & b_d & b_{d-1}    & \cdots  &    b_1     &    b_0  &      &      \\
                              &     &            &         & \vdots     &         &      &      \\
                              &     &            &    b_d  &   b_{d-1}  & \cdots  & b_1  & b_0 
                             \end{array} \right] \ \Bigg) \ , 
\end{equation}
Here, $a_d = C_0 \cdot \prod_{i=m+1}^d \delta_i$ and $b_d = C_1 \cdot \prod_{j=n+1}^d \eta_j$.  

Now dehomogenize $F(X,Y)$ and $G(X,Y)$, setting $z = X/Y$, obtaining 
\begin{eqnarray*} 
f(z) & = & a_d z^d + a_{d-1} z^{d-1} + \cdots + a_0 \ = \ a_d \cdot \prod_{i=1}^d (z-\alpha_i) \ , \\
g(z) & = & b_d z^d + b_{d-1} z^{d-1} + \cdots + b_0 \ = \ b_d \cdot \prod_{j=1}^d (z-\beta_j) \ .
\end{eqnarray*} 
where $z = X/Y$ and now $\alpha_1, \ldots, \alpha_d, \beta_1, \ldots, \beta_d \in K$.  Evidently 
\begin{equation*}
\begin{array}{ll}
\alpha_1 = \sigma_1, \ldots, \alpha_m = \sigma_m, & \quad \alpha_{m+1} = 1/\delta_{m+1} \ldots, \alpha_d = 1/\delta_d \ , \\
\beta_1 = \tau_1, \ldots, \beta_n = \tau_n, & \quad \beta_{n+1} = 1/\eta_{n+1} \ldots, \beta_d = 1/\eta_d \ .
\end{array}
\end{equation*}

The resultant $\Res(f,g)$ is given by the same determinant (\ref{ResultantFormula}) as $\Res(F,G)$.  
Since $a_d, b_d \ne 0$, a well-known formula for the resultant (see [L], Proposition 10.3) gives 
\begin{equation*} 
\Res(f,g) \ = \ (a_d)^d \cdot (b_d)^d \cdot \prod_{i,j = 1}^d (\alpha_i-\beta_j) \ .
\end{equation*}  
Inserting the above values for $a_d$, $b_d$ and the $\alpha_i$, $\beta_j$, then simplifying, we see that 
\begin{eqnarray*}
|\Res(F,G)| \ = \ |\Res(f,g)| & = & |C_0|^d \cdot |C_1|^d \cdot \prod_{i=1}^m \prod_{j=1}^n |\sigma_i - \tau_j| 
					\cdot \prod_{i=m+1}^d \prod_{j=1}^n | 1 - \delta_i \tau_j|\\
    &  & \qquad \quad 
             \cdot \prod_{i=1}^m \prod_{j=n+1}^d | \sigma_i \eta_j - 1|  
   \cdot \prod_{i=m+1}^d \prod_{j=n+1}^d | \eta_j - \delta_i| \ .
\end{eqnarray*} 
Here 
\begin{equation*}
\left\{ \begin{array}{ll}
   |\sigma_i - \tau_j| = \|\alpha_i, \beta_j\|       & \text{ for $i = 1, \ldots, m$, $j= 1, \ldots, n$;} \\
   |1 - \delta_i \tau_j| = 1 = \|\alpha_i, \beta_j\| &  \text{ for $i = m+1, \ldots, d$, $j= 1, \ldots, n$;} \\
   |\sigma_i \eta_j - 1| = 1 = \|\alpha_i, \beta_j\| &  \text{ for $i = 1, \ldots, m$, $j= n+1, \ldots, d$;  } \\
   |\eta_j - \delta_i| = \|\alpha_i, \beta_j\|       &  \text{ for $i = m+1, \ldots, d$, $j= n+1, \ldots, d$.}
    \end{array} \right.
\end{equation*} 
Thus $|\Res(\varphi)| = |\Res(F,G)| = |C_0|^d \cdot |C_1|^d \cdot \prod_{i,j= 1}^d \|\alpha_i,\beta_j\|$. 
\end{proof} 

\begin{corollary} \label{ResultantBoundCor}
Let $\varphi(z) \in K(z)$ have degree $d \ge 1$.  
Then $\GPR(\varphi) \ge |\Res(\varphi)|$  and $\GIR(\varphi)^d \cdot B_0(\varphi) \ge |\Res(\varphi)|$. 
In particular  $\GIR(\varphi)\ge |\Res(\varphi)|^{1/d}$ and  $B_0(\varphi) \ge |\Res(\varphi)|$. 
\end{corollary} 

\begin{proof}  Recall that $\GPR(\varphi)$, 
$\GIR(\varphi)$, $B_0(\varphi)$, and $|\Res(\varphi)|$ are invariant under pre- and post-
composition of $\varphi$ with elements of $\GL_2(\cO)$.  

\smallskip
To show that $\GPR(\varphi) \ge |\Res(\varphi)|$, put $R = \GPR(\varphi)$ 
and fix $Q \in \varphi^{-1}(\{\zeta_G\})$ with $\diam_G(Q) = R$.   
Choose $\gamma_1 \in \GL_2(\cO)$ so that $\gamma_1(\zeta_{0,R}) = Q$; after replacing $\varphi$ with $\varphi \circ \gamma_1$
we can assume that $Q = \zeta_{0,R}$.  There are at most $d$ directions $\vv \in T_P$ for which
$\varphi(\cB_Q(\vv)^-) = \BPP^1_K$, since such a direction must contain a solution to 
$\varphi(\alpha)) = 0$.  Similarly, for each $\vw \in T_{\zeta_G}$, 
there are at most $d$ directions $\vv \in T_Q$ for which $\varphi_*(\vv) = \vw$.  
Since the map $\varphi_*: T_Q \rightarrow T_{\zeta_G}$ is surjective, we can find directions 
$\vv_1, \vv_2 \in T_Q$ satisfying the following conditions:
\begin{enumerate}
\item $\vv_1, \vv_2 \in T_Q \backslash \{\vv_\infty\}$; 
\item $\vv_1 \ne \vv_2$ and $\varphi_*(\vv_1) \ne \varphi_*(\vv_2)$;
\item $\varphi(\cB_Q(\vv_1)^-) = \cB_{\zeta_G}(\varphi_*(\vv_1))^-$ and 
             $\varphi(\cB_Q(\vv_2)^-)  = \cB_{\zeta_G}(\varphi_*(\vv_2))^-$ are balls.
\end{enumerate}
Fix $\alpha \in \PP^1(K) \cap \cB_Q(\vv_1)^-$ and $\beta \in \PP^1(K) \cap \cB_Q(\vv_2)^-$, 
and note that $\|\alpha,\beta\| = R$.  Put $A = \varphi(a)$, $B = \varphi(\beta)$.    

Since $A$ and $B$ belong to distinct tangent directions at $\zeta_G$, by (\cite{B-R}, Corollary 2.13(B))
there is a $\gamma_2 \in \GL_2(K)$ which takes the triple $(A,\zeta_G,B)$ to $(0,\zeta_G,\infty)$. 
Since $\gamma_2(\zeta_G) = \zeta_G$, and the stabilizer of $\zeta_G$ is $K^\times \cdot \GL_2(\cO)$, 
we can scale $\gamma_2$ so that it belongs to $\GL_2(\cO)$.  After replacing $\varphi$ with $\gamma_2 \circ \varphi$,
we can assume that $\varphi(\alpha) = 0$ and $\varphi(\beta) = \infty$.  By Proposition \ref{ResultantProp}
\begin{equation*}
\GPR(\varphi) \ = \ R \ = \ \|\alpha,\beta\| \ \ge \ |\Res(\varphi)| \ .
\end{equation*} 

\smallskip
To show that $\GIR(\varphi)^d \cdot B_0(\varphi) \ge |\Res(\varphi)|$, put $r = \GIR(\varphi)$ and 
choose $\gamma \in \GL_2(\cO)$ with $\gamma(\varphi(\zeta_G)) = \zeta_{0,r}$. 
After replacing $\varphi$ with $\gamma \circ \varphi$ we can assume that $\varphi(\zeta_G) = \zeta_{0,r}$.

In this setting, if $(F,G)$ is a normalized representation of $\varphi$, 
and notations are as in Proposition \ref{ResultantProp}, then $|C_0| = r = \GIR(\varphi)$ and $|C_1| = 1$.  
From the inequality
$B_0(\varphi) \ge \RP(\varphi)$ it follows that $B_0(\varphi) \ge \min_{i,j} \|\alpha_i, \beta_j\|$.  Furthermore,
$\|\alpha_i, \beta_j\| \le 1$ for all $i, j$, and $|\GIR(\varphi)| \le 1$.  

Thus, by Proposition \ref{ResultantProp}, we have 
$\GIR(\varphi)^d \cdot B_0(\varphi) \ge |C_0|^d \cdot \min_{i,j} \|\alpha_i, \beta_j\| \ \ge \ |\Res(\varphi)|$.   
Since $\GIR(\varphi) \le 1$ and $B_0(\varphi) \le 1$, the last two inequalities in the Corollary are immediate.
\end{proof}

\section{Proofs of Theorems \ref{FirstCor}, \ref{MainThm0}, and \ref{Classical_Lip}} \label{MainThmSection} 

Our main result is 

\begin{theorem}  \label{MainThm}  Let $K$ be a complete, algebraically closed nonarchimedean field,  
and let $\varphi(z) \in K(z)$ have degree $d \ge 1$.  Then
\begin{equation} \label{MainBound} 
\Lip_\Berk(\varphi) \ \le \ \max\Big( \frac{1}{\GIR(\varphi) \cdot B_0(\varphi)^d} \, , 
                                  \frac{d}{\GIR(\varphi)^{1/d} \cdot B_0(\varphi)} \Big) \ .
\end{equation} 
\end{theorem} 

This is a restatement of Theorem \ref{MainThm0} in the Introduction.  
Before giving the proof of Theorem \ref{MainThm}, we will make some reductions.
Put $B_0 = B_0(\varphi)$, and consider a ball $\cB(a,B_0)^-$.  By the definition of $B_0$, the image 
$\varphi(\cB(a,B_0)^-)$ is a ball.  In particular there is an $\alpha \in \PP^1(K)$ with 
$\alpha \notin \varphi(\cB(a,B_0)^-)$.  By choosing $\gamma_1, \gamma_2 \in \GL_2(\cO)$ 
with $\gamma_1(\alpha) = \infty$, $\gamma_2(0) = a$, 
and replacing $\varphi$ with $\Phi = \gamma_1 \circ \varphi \circ \gamma_2$,
we can arrange that that $a = 0$ and that $\Phi(D(0,B_0)^-)$ omits $\infty$.
   
This means that $\Phi(D(0,B_0)^-)$ is a disc $D(c_0,R)^-$, where $\Phi(0) = c_0$.
By the Weierstrass Preparation theorem, we can expand $\Phi(z)$ on $D(0,B_0)^-$ in the form 
\begin{equation*}
\Phi(z) \ = \ c_0 + (c_1 z + c_2 z^2 + \cdots + c_n z^n) \cdot U(z)
\end{equation*} 
where $1 \le n \le d$ is the number of solutions to $\Phi(z) = c_0$ in $B(0,B_0)^-$, and 
$U(z) = 1 + u_1 z + u_2 z^2 + \cdots$ is a unit power series converging on $D(0,B_0)^-$; 
here $|u_i| \le 1/B_0^i$ for each $i$.  

Put 
\begin{equation*}
f_\Phi(r) \ = \ \max_{1 \le k \le n} |c_k| r^k \ . 
\end{equation*} 
By the theory of Newton polygons, for each $r \in |K^{\times}|$ with $0 < r < B_0$, we have $\Phi(D(a,r)) = D(c_0,f(r))$, 
so 
\begin{equation} \label{MapFormula}
\Phi(\zeta_{0,r}) \ = \ \zeta_{c_0,f_\Phi(r)}
\end{equation} 
By the continuity of the action of $\Phi$ on $\BPP^1_K$, 
(\ref{MapFormula}) holds for all $r \in [0,B_0]$, and $f_\Phi(B_0) = R$.


We will prove Theorem \ref{MainThm} by applying Corollary \ref{SupCor}, and dealing with five cases: 
four cases dealing with radial paths $[0,B_0]$ contained in  balls $\cB(0,B_0)^-$ according as
\begin{equation*}
\left\{\begin{array}{l} 
                \text{$|c_0| \le 1$ and $f_\Phi(B_0) \le 1$ ,} \\
                \text{$|c_0| \le 1$ and $f_\Phi(B_0) > 1$ ,} \\
                \text{$|c_0| > 1$ and $f_\Phi(B_0) \le |c_0|$ ,} \\
                \text{$|c_0| > 1$ and $f_\Phi(B_0) > |c_0|$ ,} 
\end{array} \right.
\end{equation*}
and one case dealing with radial paths $[\xi,\zeta_G]$ in the central ball 
\begin{equation*}
\cB_\rho(\zeta_G,-\log(B_0))^- \ = \ \{x \in \BPP^1_K : \diam_G(x) \ge B_0 \} \ .
\end{equation*} 

\smallskip
Case 1 is covered by the following Proposition: 

\begin{proposition} \label{Case1}  Let $\Phi(z) \in K(z)$ have degree $d \ge 1$; write $B_0 = B_0(\Phi)$, 
put $c_0 = \Phi(0)$, and assume $\Phi$ has no poles in $D(0,B_0)^-$, so $\Phi(D(0,B_0)^-) = D(c_0,R)^-$ is a disc. 
Suppose $|c_0| \le 1$ and $R \le 1$. 
Then the restriction of $\Phi$ to $[0,\zeta_{0,B_0}]$ satisfies 
\begin{equation*}
\Lip_\Berk\big(\Phi\vert_{[0,\zeta_{0,B_0}]}\big) \ \le \ \frac{d}{B_0} \ .
\end{equation*} 
\end{proposition} 

\begin{proof} 
As in Lemma \ref{fprimeMonoLemma} 
we can partition $[0,B_0]$ into subintervals $[r_i,r_{i+1}]$
where $0 =  r_1 < \cdots < r_{\ell+1} = B_0$, such that on $[r_{i-1},r_i]$ we have  
\begin{equation*}
f_\Phi(r) \ = \ f_i(r) \ = \ |c_{k(i)}| \cdot r^{k(i)} \ .
\end{equation*} 
Write $f_\Phi^{\prime}(r)$ for the right-derivative of $f_\Phi(r)$ on $[0,B_0)$.  
By Lemma \ref{fprimeMonoLemma}, $f_\Phi^{\prime}(r)$ is non-decreasing.  Hence
\begin{equation} \label{fprimeLimit1}
\sup_{r \in [0,B_0)} f_\Phi^\prime(r) \ = \ \lim_{r \rightarrow B_0^-} f_\ell^\prime(r) 
\ = \ k(\ell) \cdot |c_{k(\ell)}| \cdot B_0^{k(\ell) - 1} \ = \ \frac{k(\ell) \cdot f_\Phi(B_0)}{B_0} \ .
\end{equation}  

Since $|c_0| \le 1$ and $f_\Phi(B_0) \le 1$, we have $\Phi(D(0,B_0)^-) \subset D(0,1)$, 
so $F_\Phi(r) = f_\Phi(r)$ for all $r \in [0,B_0]$. 
Using the inequalities $k(\ell) \le n \le d$ and $f_\Phi(B_0) \le 1$ we conclude 
from (\ref{fprimeLimit1}) that 
\begin{equation*} 
\Lip_\Berk\big(\Phi\vert_{[0,B_0)}\big) \ = \ \sup_{r \in [0,B_0)} f_\Phi^\prime(r)\ \le \ \frac{d}{B_0} \ .
\end{equation*} 
\end{proof} 

To deal with Case 2, we will need several lemmas.   The first is an elementary maximization bound 
from Calculus:

\begin{lemma} \label{CalculusLemma}  Let $H \ge 1$, and put  $g(x) = x \cdot H^{1/x}$ for $x > 0$.  
Then for each closed interval $[a,b] \subset (0,\infty)$, 
\begin{equation*}
\max_{x \in [a,b]} g(x) \ = \ \max\big(g(a),g(b)\big) \ .
\end{equation*} 
\end{lemma} 

\begin{proof}  If $H = 1$ then $g(x) = x$ and the result is trivial.  
If $H > 1$, then $g^{\prime}(x) = (1-\ln(H)/x) \cdot H^{1/x}$ and $g^{\prime \prime}(x) = (\ln(H))^2/x^3 \cdot H^{1/x}$, 
so $g(x)$ is convex up for $x > 0$, and its unique minimum 
is at $x = \ln(H)$.  Thus the maximum value of $g(x)$ on $[a,b]$  
is achieved at an endpoint.
\end{proof}

The second is a bound for $\lim_{|z| \rightarrow 1^-} |\Phi(z)|$.
Recall that if $P,Q$ are distinct points in $\BPP^1_K$, 
the annulus $\Ann(P,Q)$ is the component of $\BPP^1_K \backslash \{P,Q\}$ containing $(P,Q)$.   

\begin{lemma} \label{LimitLemma} 
Let $\Phi(z) \in K(z)$ have degree $d \ge 2$; write $B_0 = B_0(\Phi)$.  If $\Phi(\zeta_G) = \zeta_{a,R}$, then 
\begin{equation} \label{Want}
\lim_{\substack{|z| \rightarrow 1^- \\ z \in K}} |\Phi(z)| \ = \ \max(|a|,R) \ .
\end{equation} 
\end{lemma} 

\begin{proof}
Let $\vv_0 \in T_{\zeta_G}$ be the tangent direction towards $0$, 
and put $\vw = \Phi_*(\vv_0) \in T_{\zeta_{a,R}}$.
Fix a point $b \in B_{\zeta_{a,R}}(\vw)^- \cap \PP^1(K)$.  By (\cite{B-R}, Corollary 9.21 and Lemma 9.45), 
there are points $X \in (\zeta_G,0)$, $Y \in (\zeta_{a,R},b)$ 
such that $\Phi$ maps $[\zeta_G,X]$ homeomorphically onto $[\zeta_{a,R},Y]$,
and for each $x \in [\zeta_G,X]$,  
 $\Phi$ maps $\Ann(\zeta_G,x)$ onto $\Ann(\zeta_{a,R},\Phi(x))$.  
Put $r = \diam_\infty(x)$, $S = \diam_\infty(\Phi(x))$. 
When $r \rightarrow 1^-$, then $\Phi(x) \rightarrow \zeta_{a,R}$ and $S \rightarrow R$.
Note that
\begin{equation*}
K \cap \Ann(\zeta_G,x) \ = \ \{ z \in K : r < |z| < 1 \} \ .
\end{equation*} 

Consider the possibilities for $\Phi\big(K \cap \Ann(\zeta_G,x)\big)$.
 If $|a| \le R$, then $\zeta_{a,R} = \zeta_{0,R}$.  
In this situation, if $\vw \in T_{\zeta_{0,R}}$ points towards $0$, then for $r$ near enough $1$, 
$\Phi\big(K \cap \Ann(\zeta_G,x)\big) = \{ z \in K : S < |z| < R \}$.  If $\vw$ points towards $\infty$,
then for $r$ near enough $1$, $\Phi\big(K \cap \Ann(\zeta_G,x)\big) = \{ z \in K : R < |z| < S \}$.
Otherwise, $\Phi\big(K \cap \Ann(\zeta_G,x)\big) \subset D(b,R)^- \subset \{ z \in K : |z| = R \}$.  
In any case, 
\begin{equation*} 
\lim_{\substack{|z| \rightarrow 1^- \\ z \in K}} |\Phi(z)| \ = \ R \ = \ \max(|a|,R) \ .
\end{equation*}
If $|a| > R$, put $R_0 = |\,a|$.  Regardless of the direction $\vw$, when $r$ is close enough to $1$ we will have 
$\Phi\big(K \cap \Ann(\zeta_G,x)\big) \subset D(a,R_0)^- \subset \{z \in K : |z| = |a| \}$.  Thus
\begin{equation*} 
\lim_{\substack{|z| \rightarrow 1^- \\ z \in K}} |\Phi(z)| \ = \ |a| \ = \ \max(|a|,R) \ .
\end{equation*}
Hence (\ref{Want}) holds in all cases.
\end{proof} 

The third is a bound for $f_\Phi(B_0)$: 

\begin{lemma} \label{GrowthBound2} 
Let $\Phi(z) \in K(z)$ have degree $d \ge 2;$ write $B_0 = B_0(\Phi)$.  
Assume that $\Phi(0) = 0$, and that $\Phi$ has no poles in $D(0,B_0)^-$,   
so $\Phi(D(0,B_0)^-) = D(0,f_\Phi(B_0))^-$ is a disc, and         
$f_\Phi(r) = \diam_\infty\big(\Phi(\zeta_{0,r})\big)$ 
is increasing for $0 \le r \le B_0$.  
Suppose $\Phi$ has $n \ge 1$ zeros in $D(0,B_0)^-$, 
and $m \ge 0$ poles in $D(0,1)^- \backslash D(0,B_0)^-$ $($counting multiplicities$)$.   
Then   
\begin{equation*}
f_\Phi(B_0) \ \le \ \frac{B_0^{n-m}} {\GIR(\Phi)}  \ .
\end{equation*} 
\end{lemma}

\begin{proof}   Since $\Phi(0) = 0$ and $\Phi$ has no poles in $D(0,B_0)^-$, 
for each $0 < r \le B_0$, $\Phi(D(0,r)^-)$ is a disc $D(0,f_\Phi(r))^-$; 
clearly $f_\Phi(r)$ is increasing with $r$. 

 We can write
\begin{equation} \label{FacF1}
\Phi(z) \ = \ C \cdot \frac{\prod_{i=1}^N (z-\alpha_i)}{\prod_{j=1}^M (z-\beta_j)}
\end{equation}
where $C \ne 0$ is a constant, $\alpha_1, \ldots, \alpha_N$ 
are the zeros of $\Phi$ in $K$ (listed with multiplicity),
and $\beta_1, \ldots, \beta_M$ are the poles of $\Phi$ in $K$ 
(listed with multiplicity).  Since $\deg(\Phi) = d$, $\max(N,M) = d$. 
Without loss, we can assume that $0 = |\alpha_1| \le |\alpha_2| \le \cdots \le |\alpha_N|$.

Since $\Phi$ has $n$ zeros in $D(0,B_0)^-$ and $\Phi(D(0,B_0)^-) = D(0,f_\Phi(B_0))^-$,  
(\ref{FacF1})  gives 
\begin{equation}  \label{CValue21}
f_\Phi(B_0) \ = \ \lim_{|z| \rightarrow B_0^-} |\Phi(z)| 
\ = \ |C| \cdot \frac{B_0^n \cdot \prod_{i=n+1}^N \max(B_0,|\alpha_i|)}{\prod_{j=1}^M \max(B_0,|\beta_j|)}  \ . 
\end{equation} 
Also, writing $\Phi(\zeta_G) = \zeta_{a,R}$,   
by Lemma \ref{LimitLemma} and formula (\ref{diamGformula}) we have  
\begin{equation} \label{LimF11} 
\lim_{\substack{|z| \rightarrow 1^- \\ z \in K}} |\Phi(z)| \ = \ \max(|a|,R) 
\ \le \ \frac{\max(1,|a|,R)^2}{R} \ = \ \frac{1}{\GIR(\Phi)} \ .
\end{equation} 
Using (\ref{FacF1}) to evaluate $\lim_{|z| \rightarrow 1^-}|\Phi(z)|$ in (\ref{LimF11}), we see that  
\begin{equation} \label{CValue11}
|C| \cdot \frac{\prod_{i=1}^M \max(1,|\alpha_i|)}{\prod_{j=1}^N \max(1,|\beta_j|)} \ \le \ \frac{1}{\GIR(\Phi)} \ . 
\end{equation}

Using (\ref{CValue11}) to eliminate $|C|$ in (\ref{CValue21}), 
and recalling that $\Phi$ has no poles in $B(0,B_0)^-$ and $m$ poles in $B(0,1)^- \backslash B(0,B_0)^-$, we obtain 
\begin{equation*} 
f_\Phi(B_0) \ \le \ \frac{1}{\GIR(\Phi)} \cdot \frac{B_0^n \cdot \prod_{B_0 \le |\alpha_i| < 1} |\alpha_i|}{\prod_{B_0 \le |\beta_j| < 1} |\beta_j| }
\ \le \ \frac{B_0^{n-m}}{\GIR(\Phi)} \ . 
\end{equation*} 
\end{proof}

Case 2 is covered by the following Proposition: 

\begin{proposition} \label{Case2} 
Let $\Phi(z) \in K(z)$ have degree $d \ge 1$; write $B_0 = B_0(\Phi)$, 
and put $c_0 = \Phi(0)$.  Assume $\Phi$ has no poles in $D(0,B_0)^-$, so $\Phi(D(0,B_0)^-) = D(c_0,R)^-$ 
is a disc.   Suppose $|c_0| \le 1$ and $R > 1$. 
Then the restriction of\, $\Phi$ to $[0,\zeta_{0,B_0}]$ satisfies 
\begin{equation} \label{Case2Formula} 
\Lip_\Berk\big(\Phi\vert_{[0,\zeta_{0,B_0}]}\big) \ \le \ 
\max\big( \frac{1}{\GIR(\Phi) \cdot B_0^d}, \frac{d}{\GIR(\Phi)^{1/d} \cdot B_0} \big) \ .
\end{equation} 
\end{proposition} 

\begin{proof}  Since  $\gamma_{c_0}(z) := z-c_0 \in \GL_2(\cO)$, 
after replacing $\Phi(z)$ with $\gamma_{c_0} \circ \Phi(z) = \Phi(z) - c_0$,
we can assume that $\Phi(0) = 0$.   
By the Weierstrass Preparation theorem, we can expand $\Phi(z)$  in $D(0,B_0)^-$ as 
\begin{equation*}
\Phi(z) \ = \ (c_1 z + \cdots + c_n z^n) \cdot U(z)  
\end{equation*}
where $U(z)$ is a unit power series.  Hence  
\begin{equation*}
f_\Phi(r) \ = \ \max_{1 \le k \le n}( |c_k| r^k) \ , \qquad 
F_\Phi(r) \ = \ \left\{ \begin{array}{ll} f_\Phi(r) & \text{if $f_\Phi(r) \le 1$ \ , } \\
                                         1/f_\Phi(r) & \text{if $f_\Phi(r) \ge 1$ \ .}
                       \end{array} \right.
\end{equation*}
By Lemma \ref{fprimeMonoLemma}, there is a partition $0 = r_1 < \cdots < r_{\ell+1} = B_0$ of $[0,B_0]$
such that for each subinterval $[r_i,r_{i+1}]$ there is an index $k(i)$ for which $f_\Phi(r) = |c_{k(i)}| r^{k(i)}$.
After inserting an extra partition point if necessary, we can assume there is an $i_0$ 
for which $f_\Phi(r_{i_0}) = 1$.   
The right-derivative $f_\Phi^\prime(r)$ is non-decreasing on $[0,B_0]$,
and $k(i-1) \le k(i)$ for each $i =2, \ldots, \ell$.

We will now bound the absolute value of the right-derivative $F_\Phi^{\prime}(r)$ on $[0,B_0)$. 
Since $F_\Phi(r) = f_\Phi(r)$ on $[0,r_{i_0})$,  
\begin{equation*}
\sup_{r \in [0,r_{i_0})} |F_\Phi^\prime(r)| 
\ = \ \lim_{r \rightarrow r_{i_0}^-} f_\Phi^{\prime}(r) \ \le \ f_\Phi^\prime(r_{i_0})
\ = \ \frac{k(i_0)}{r_{i_0}} \ .   
\end{equation*}
For each $i \ge i_0$, on the interval $[r_i,r_{i+1})$  
\begin{equation*}
|F_\Phi^\prime(r)| \ = \ \left|-\frac{f_\Phi^{\prime}(r)}{f_\Phi(r)^2}\right| 
\ = \ \frac{k(i)}{r} \cdot \frac{1}{f_\Phi(r)} 
\ \le \ \frac{k(i)}{r_i \cdot f_\Phi(r_i)} \ \le \ \frac{k(i)}{r_i}\ . 
\end{equation*}  
Thus by Corollary \ref{DerivCor}, 
\begin{eqnarray} 
\Lip_\Berk(\Phi\vert_{[0,B_0]}) \ = \ \sup_{r \in [0,B_0)} |F_\Phi^{\prime}(r)| 
& = & \max_{i_0 \le i \le \ell} \ \frac{k(i)}{r_i f_\Phi(r_i)} \label{FprimeBasicIneq1} \\
& \le & \max_{i_0 \le i \le \ell} \frac{k(i)}{r_i} \ . \label{FprimeBasicIneq2}
\end{eqnarray} 
Here the second and third expressions in (\ref{FprimeBasicIneq1}) are equal 
by the right-continuity of $f_\Phi^{\prime}(r)$. 

For each $k = 1, \ldots, n$ with $c_k \ne 0$, let $u_k > 0$ be the unique solution to $|c_k| r^k = 1$.  
For brevity, write $G_0 = \GIR(\Phi)$. 
By the nonarchimedean Maximum Modulus principle and Lemma \ref{GrowthBound2}, 
we have $|c_k| B_0^k \le f_\Phi(B_0) \le B_0^{n-m}/G_0$,
where $n$ is the number of zeros of $\Phi$ in $D(0,B_0)^-$ and $m$ is the number of poles
of $\Phi$ in $D(0,1)^- \backslash D(0,B_0)^-$.  Hence $|c_k| \le B_0^{n-m-k}/G_0$, so 
$1 = |c_k| (u_k)^k \le B_0^{n-m-k} \cdot (u_k)^k/G_0$, and
\begin{equation} \label{uIneq}
\frac{1}{u_k} \ \le \ B_0^{-1} \cdot \Big( \frac{B_0^{n-m}}{G_0} \Big)^{1/k} \ .
\end{equation} 
For each $i \ge i_0$, we have $u_{k(i)} \le r_i$.   
Using  Corollary \ref{DerivCor} and (\ref{FprimeBasicIneq2}), (\ref{uIneq}), we see that 
\begin{eqnarray}
\Lip_\Berk\big(\Phi\vert_{[0,\zeta_{0,B_0}]}\big) & = & \sup_{r \in [0,B_0)} |F_\Phi^{\prime}(r)| 
\ \le \ \max_{i_0 \le i \le \ell} \frac{k(i)}{r_i}  
\ \le \ \max_{i_0 \le i \le \ell} k(i) \cdot B_0^{-1} \cdot \Big( \frac{B_0^{n-m}}{G_0} \Big)^{1/k(i)}  \notag \\
& \le &  B_0^{-1} \cdot\max_{1 \le k \le n} k \cdot \Big( \frac{B_0^{n-m}}{G_0} \Big)^{1/k} \ . 
\label{fIntermediate}
\end{eqnarray} 

However, we want a bound for $\Lip_\Berk\big(\Phi\vert_{[0,\zeta_{0,B_0}]}\big)$ independent of $n$ and $m$.
By the discussion above $1 \le k \le n \le d$ and $0 \le m \le d$.  
We need only consider pairs $(n,m)$ for which $B_0^{n-m}/G_0 > 1$, 
since by assumption $1 < f_\Phi(B_0)$ and Lemma \ref{GrowthBound2} gives $f_\Phi(B_0) \le B_0^{n-m}/G_0$.
Letting $(k,n,m)$ range over all 
triples of integers satisfying these conditions 
we see that 
\begin{equation} \label{kmnFormula}
\Lip_\Berk\big(\Phi\vert_{[0,\zeta_{0,B_0}]}\big) 
\ \le \ \max_{\substack{ 0 \le m \le d \\ 1 \le n \le d \\ B_0^{n-m}/G_0 > 1}} 
 \Big( \ B_0^{-1} \cdot \max_{1 \le k \le n} k \cdot \Big( \frac{B_0^{n-m}}{G_0} \Big)^{1/k} \ \Big) \ .
\end{equation}

We will now bound the right side of (\ref{kmnFormula}).
Fixing $n$ and $m$ with $B_0^{n-m}/G_0 > 1$, and taking $H = B_0^{n-m}/G_0$ in Lemma \ref{CalculusLemma},
shows that 
\begin{equation*}
B_0^{-1} \cdot \max_{1 \le k \le n} k \cdot \Big( \frac{B_0^{n-m}}{G_0} \Big)^{1/k} 
\ = \ \max \Big( \ \frac{B_0^{n-m-1}}{G_0},  n \Big(\frac{B_0^{-m}}{G_0} \Big)^{1/n} \ \Big) \ .
\end{equation*}
Inserting this in (\ref{kmnFormula}), interchanging the order of the maxima, and dropping the condition
$B_0^{n-m}/G_0 > 1$ gives 
\begin{equation} \label{kmnFormula2}
\Lip_\Berk\big(\Phi\vert_{[0,\zeta_{0,B_0}]}\big) \ \le \ 
\max \Big( \ \max_{\substack{ 0 \le m \le d \\ 1 \le n \le d}} \frac{B_0^{n-m-1}}{G_0} , 
\max_{\substack{ 0 \le m \le d \\ 1 \le n \le d }} 
                n \Big(\frac{B_0^{-m}}{G_0} \Big)^{1/n} \ \Big)  \ .
\end{equation}
The first inner maximum in (\ref{kmnFormula2}) is $B_0^{-d}/G_0$, achieved when $n = 1$ and $m = d$.  
For the second inner maximum, fixing $m$ and taking $H = B_0^{-m}/G_0$ in Lemma \ref{CalculusLemma} gives 
\begin{equation*}
\max_{ 1 \le n \le d } n \Big(\frac{B_0^{-m}}{G_0} \Big)^{1/n} 
 \ = \ \max \Big( \frac{B_0^{-m}}{G_0}\ , d \Big(\frac{B_0^{-m}}{G_0} \Big)^{1/d} \Big)
\end{equation*} 
The maximum of this for $0 \le m \le d$ is attained when $m = d$, and is 
\begin{equation*}
\max \Big( \frac{1}{ G_0 \cdot B_0^d} \ , \frac{d}{G_0^{1/d} \cdot B_0}  \Big)
\end{equation*}
Combining these results gives (\ref{Case2Formula}).
\end{proof}
 
Case 3 is covered by the following result:

\begin{proposition} \label{Case3}
Let $\Phi(z) \in K(z)$ have degree $d \ge 1$; write $B_0 = B_0(\Phi)$, 
put $c_0 = \Phi(0)$, and assume that $\Phi(D(0,B_0)^-) = D(c_0,R)^-$.  
Suppose $|c_0| > 1$ and $R \le |c_0|$. 
Then the restriction of $\Phi$ to $[0,\zeta_{0,B_0}]$ satisfies 
\begin{equation*}
\Lip_\Berk\big(\Phi\vert_{[0,\zeta_{0,B_0}]}\big) \ \le \ \frac{d}{B_0} \ .
\end{equation*} 
\end{proposition} 

\begin{proof} 
Partition $[0,B_0]$ by taking $0 =  r_1 < \cdots < r_{\ell+1} = B_0$ so that  
\begin{equation*}
f_\Phi(r) \ = \ f_i(r) \ := \ |c_{k(i)}| \cdot r^{k(i)} 
\end{equation*}
on $[r_i,r_{i+1})$. Write $f_\Phi^{\prime}(r)$ for the right-derivative of $f_\Phi(r)$.  
Just as in Proposition \ref{Case1}, 
\begin{equation} \label{fprimeLimit2}
\sup_{r \in [0,B_0)} f_\Phi^\prime(r) \ = \ \lim_{r \rightarrow B_0^-} f_\ell^\prime(r) 
\ = \ k(\ell) \cdot |c_{k(\ell)}| \cdot B_0^{k(\ell) - 1} \ = \ \frac{k(\ell) \cdot f_\Phi(B_0)}{B_0} \ .
\end{equation}   
 
Since $|c_0| > 1$ and $f_\Phi(B_0) \le |c_0|$, we have $\Phi(D(0,B_0)^-) \subset D(c_0,|c_0|)^-$,
hence $F_\Phi(r) = f_\Phi(r)/|c_0|$ and  $F_\Phi^{\prime}(r) = f_\Phi^{\prime}(r)/|c_0|$ for all  $r \in [0,B_0)$.   
Using that $k(\ell) \le n \le d$ and $f_\Phi(B_0) \le |c_0|$ we conclude 
from (\ref{fprimeLimit2}) that 
\begin{equation*} 
\Lip_\Berk\big(\Phi\vert_{[0,B_0)}\big) \ = \ \sup_{r \in [0,B_0)} F_\Phi^\prime(r)\ \le \ \frac{d}{B_0} \ .
\end{equation*} 
\end{proof} 

Case 4 reduces to Case 2 by a trick:  

\begin{proposition} \label{Case4}
Let $\Phi(z) \in K(z)$ have degree $d \ge 1$; write $B_0 = B_0(\Phi)$, 
put $c_0 = \Phi(0)$, and assume $\Phi(D(0,B_0)^-) = D(c_0,R)^-$.  
Suppose $|c_0| > 1$ and $R > |c_0|$.     
Then the restriction of\, $\Phi$ to $[0,\zeta_{0,B_0}]$ satisfies 
\begin{equation*}
\Lip_\Berk\big(\Phi\vert_{[0,\zeta_{0,B_0}]}\big) \ \le \ 
\max\big( \frac{1}{\GIR(\Phi) \cdot B_0^d}, \frac{d}{\GIR(\Phi)^{1/d} \cdot B_0} \big) \ .
\end{equation*} 
\end{proposition} 

\begin{proof}  
Since $f_\Phi(r)$ is continuous and monotonic,
with $f_\Phi(0) = 0$ and $f_\Phi(B_0) > |c_0| > 1$, there is a unique $0 < R < B_0$ with $f_\Phi(R) = |c_0|$.  For this $R$,
the theory of Newton polygons shows that $\Phi(D(0,R)) = D(c_0,|c_0|) = D(0,|c_0|)$,  so there is an $\alpha \in D(0,R)$
for which $\Phi(\alpha) = 0$.  Write $\gamma_\alpha(z) = z + \alpha \in \GL_2(\cO)$, and put 
\begin{equation*}
\Psi(z) \ = \ \Phi(z+\alpha) \ = \ (\Phi \circ \gamma_\alpha)(z) \ .
\end{equation*} 
By construction $\Psi(0) = 0$,   
so $\Psi$ satisfies the conditions of Proposition \ref{Case2}.  
Since $\gamma_\alpha(z) \in \GL_2(\cO)$ we have $\GIR(\Psi) = \GIR(\Phi)$ and $B_0(\Psi) = B_0(\Phi) = B_0$.  

We will prove Proposition \ref{Case4} by showing that  $\Lip_\Berk(\Phi_{[0,B_0]}) \le \Lip_\Berk(\Psi_{[0,B_0]})$ 
and applying Proposition \ref{Case2} to $\Psi$. 

For each $r$ with $R < r < B_0$,
we have $\Psi(\cD(0,r)^-) = \Phi(\cD(0,r)^-)$, so  $f_\Psi(r) = f_\Phi(r)$ for $R \le r \le B_0$. 
As usual, we can write $\Phi(z) = c_0 + (c_1 z + \cdots + c_n z^n) \cdot U(z)$, where $U(z)$
is a unit power series converging on $\cD(0,B_0)^-$;  then 
\begin{equation*}
f_\Phi(r) \ = \ \max_{1 \le k \le n} |c_k| r^k 
\end{equation*}  
for each $r \in [0,B_0]$.  
Likewise we can write $\Psi(z) = (C_1 z + \cdots + C_N z^N) \cdot W(z)$, where $W(z)$
is a unit power series converging on $D(0,B_0)^-$;  and 
\begin{equation*}
f_\Psi(r) \ = \ \max_{1 \le k \le N} |C_k| r^k
\end{equation*}  
for each $r \in [0,B_0]$.  

Partition $[0,B_0]$ simultaneously for $\Phi$ and $\Psi$, 
choosing $0 =  r_1 < \cdots < r_{\ell+1} = B_0$ so that for each $i = 1, \ldots, \ell$ there are indices $1 \le j(i) \le n$,
$1 \le k(i) \le N$  such that on $[r_i,r_{i+1})$ we have  
\begin{equation*}
f_\Phi(r) \ = \ |c_{j(i)}| \cdot r^{j(i)}  \ , \qquad f_\Psi(r) \ = \ |C_{k(i)}| \cdot r^{k(i)} \ .
\end{equation*}
After refining the partition if necessary, we can assume there are indices $i_0$, $i_1$ such that 
$f_\Psi(r_{i_0}) = 1$ and $f_\Psi(r_{i_1}) = f_\Phi(r_{i_1}) = |c_0|$ (evidently $r_{i_1} = R$).  
Clearly $i_0 < i_1$, since $f_\Psi$ is monotonic.  

We claim that $j(i) = k(i)$ for $i = i_1, \cdots, \ell$.  
To see this, note first that for each $r$ with $r_{i_1} \le r \le B_0$, we have $f_\Phi(r) = f_\Psi(r)$. 
For each $r \in |K^{\times}|$ with $r_i < r < r_{i+1}$, 
and each $w \in K$ with $|w| \le f_{\Psi}(r)$, the theory of Newton polygons shows that $\Phi(z) = w$ has 
$j(i)$ solutions in $D(0,r)$, counting multiplicities. Similarly $\Psi(z) = w$ has $k(i)$  solutions in $D(0,r)$
counting multiplicities.  But $\Phi(z) = w$ if and only if $\Psi(z-\alpha) = w$. 
Since $|\alpha| = r_{i_1} \le r$, we have $|z - \alpha| \le r$ if and only if $|z| \le r$.
Hence $j(i) = k(i)$.  
 
We have   
\begin{equation*}
F_\Phi(r) \ = \ \left\{ \begin{array}{ll} f_\Phi(r)/|c_0|^2 & \text{if $r \in [0,r_{i_1})$ \ , } \\
                                         1/f_\Phi(r) & \text{if $r \in [r_{i_1},B_0)$ \ ,}
                       \end{array} \right.
\end{equation*}
and 
\begin{equation*}
F_\Psi(r) \ = \ \left\{ \begin{array}{ll} f_\Psi(r) & \text{if $r \in [0,r_{i_0})$ \ , } \\
                                         1/f_\Psi(r) & \text{if $r \in [r_{i_0},B_0)$ \ .}
                       \end{array} \right.
\end{equation*}
Write $f_\Phi^{\prime}(r)$ for the right-derivative of $f_\Phi(r)$ on $[0,B_0)$, 
and $F_\Phi^{\prime}(r)$ for the right-derivative of $F_\Phi(r)$.
Noting that $|c_0| = f_\Phi(r_{i_1}) = |c_{k(i_1)}| (r_{i_1})^{j(i_1)}$  
and recalling from Lemma \ref{fprimeMonoLemma} that $f_\Phi^{\prime}(r)$ is non-decreasing with $r$, we see that 
\begin{equation*} 
\sup_{r \in [0,r_{i_1})} |F_\Phi^{\prime}(r)| \ \le \ \frac{f_\Phi^{\prime}(r_{i_1})}{|c_0|^2} 
\ = \ \frac{k(i_1) |c_{j(i_1)}| (r_{i_1})^{j(i_1)-1}}{|c_0|^2} \ = \  \frac{j(i_1)}{r_{i_1}f_\Phi(r_{i_1})} \ . 
\end{equation*}  
For each $i = i_1, \cdots, \ell$, on $[r_i,r_{i+1})$ we have $F_\Phi^{\prime}(r) = -f_\varphi^{\prime}(r)/(f_\varphi(r))^2$, so 
\begin{equation*}
\sup_{r \in [r_i,r_{i+1})} |F_\Phi^{\prime}(r)| \ = 
|F_\Phi^{\prime}(r_i)| \ = \ \left| \frac{-f_\Phi^{\prime}(r_i)} {(f_{\Phi}(r_i))^2} \right|
\ = \ \frac{j(i)}{r_i f_\Phi(r_i)} \ .
\end{equation*}
Thus
\begin{equation*}
\Lip_\Berk(\Phi\vert_{[0,B_0]}) \ = \ \sup_{r \in [0,B_0)} |F_\Phi^{\prime}(r)| 
\ = \ \max_{i_1 \le i \le \ell} \Big( \frac{j(i)}{r_i f_\Phi(r_i)} \Big) \ .
\end{equation*} 
Since $i_0 \le i_1$, and $k(i) = j(i)$ and $f_\Phi(r_i) = f_\Psi(r_i)$ for each $i \ge i_1$, 
by applying (\ref{FprimeBasicIneq1}) with $\Phi$ replaced by $\Psi$ 
and then using the bound for $\Lip_\Berk(\Psi\vert_{[0,B_0]})$ from Proposition \ref{Case2}, we get 
\begin{eqnarray*} 
\Lip_\Berk(\Phi|_{[0,B_0]})
& = & \max_{i_1 \le i \le \ell} \Big( \frac{j(i)}{r_i f_\Phi(r_i)} \Big) 
\ \le \ \max_{i_0 \le i \le \ell} \Big( \frac{k(i)}{r_i f_\Psi(r_i)} \Big) \\
& = & \Lip_\Berk(\Psi|_{[0,B_0]})  
\ \le \ \max\Big(\ \frac{1}{\GIR(\Phi) \cdot B_0^d}\ , \frac{d}{\GIR(\Phi)^{1/d} \cdot B_0} \ \Big) \ .
\end{eqnarray*}
\end{proof} 

Case 5 (the central ball) is dealt with by the following Proposition:

\begin{proposition} \label{Case5}
Let $\Phi(z) \in K(z)$ have degree $d \ge 1$; write $B_0 = B_0(\Phi)$. 
Then 
\begin{equation*}
\Lip_\Berk\big(\Phi\vert_{[\zeta_{0,B_0},\zeta_G]}\big) \ \le \ \frac{d}{B_0} \ .
\end{equation*} 
\end{proposition} 

\begin{proof}  We use the fact that in the $\rho$-metric, along a given segment $\Phi$ locally scales distances 
by an integer $1 \le m \le d$.  To obtain the Lipschitz bound for the $d$-metric, 
we conjugate this between the $d$- and $\rho$-metrics.  
Fix a base $q > 1$ such that for each $0 < r \le 1$ we have $\rho(\zeta_G,\zeta_{0,r}) = -\log_q(r)$,
and put $E(z) = q^z$, $L(r) = \log_q(r)$.  Define $F_\Phi : (B_0,1] \rightarrow (0,1]$ by 
\begin{equation*} 
F_\Phi(r) \ = \ \diam_G\big(\Phi(\zeta_{0,r})\big) \ .
\end{equation*} 

Note that $\zeta_G = \zeta_{0,1}$.  Choose a partition $B_0 = r_1 < \cdots < r_{\ell+1} = 1$
of $[B_0,1]$ such that on each subinterval $[r_i,r_{i+1}]$, $\Phi$ has the following properties:  
\begin{enumerate}
\item $\Phi$ maps the segment $[\zeta_{0,r_i},\zeta_{0,r_{i_1}}]$ 
homeomorphically onto some radial segment;
\item there is an integer $1 \le m_i \le d$ such that  $\deg_\Phi(P) = m_i$ 
for all $P \in (\zeta_{0,r_i},\zeta_{0,r_{i+1}})$.  
\end{enumerate}
To prove the Proposition it suffices to show that   
$\Lip_\Berk\big(\Phi\vert_{[\zeta_{0,r_i},\zeta_{0,r_{i+1}}]}\big) \le d/B_0$ for each $i$. 

Fix $i$.  By (1) and (2) there is an affine function $M_i(y) = a_i y + b_i$, 
where $a_i = \pm m_i$ and $b_i \in \RR$, such that 
$\rho(\zeta_G,\Phi(\zeta_{0,r})) = M_i(L(r))$.  Hence for each $r \in [r_i,r_{i+1}]$,   
\begin{equation*}
F_\Phi(r) \ = \ E \circ M_i \circ L(r) \ .
\end{equation*} 
In particular, $F_\Phi$ is differentiable on $(r_i,r_{i+1})$.  
By the Mean Value Theorem, for each $r, s$ with $r_i \le r < s \le r_{i+1}$ 
there is an $r_* \in (r,s)$ such that 
\begin{equation*}
\frac{d(\Phi(\zeta_{0,r}),\Phi(\zeta_{0,s}))}{d(\zeta_{0,r},\zeta_{0,s})} 
\ = \ \left| \frac{F_\Phi(r)-F_\Phi(s)}{r-s} \right|\ = \ |F_\Phi^\prime(r_*)| \ ,  
\end{equation*} 
so it will be enough to show that $|F_\Phi^{\prime}(r)| \le d/B_0$ on $(r_i,r_{i+1})$.
However, this follows easily from the Chain rule:  for each $r \in (r_i,r_{i+1})$   
\begin{equation*}
|F_\Phi^\prime(r)| \ = \ \big(q^{M_i(L(r))} \cdot \ln(q)\big) \cdot |a_i| \cdot \frac{1}{r \cdot \ln(q)} 
\ = \ F_\Phi(r) \cdot \frac{m_i}{r} \ \le \ \frac{m_i}{r_i} \ \le \ \frac{d}{B_0} \ .
\end{equation*}
\end{proof} 

\begin{proof}[Proof of Theorem \ref{MainThm}]  
By Corollary \ref{SupCor}, it suffices to show that for each radial 
segment $I$ of the form $[\alpha,\xi]$ or $[\xi,\zeta_G]$, 
where $\alpha \in \PP^1(K)$ and $\diam_G(\xi) = B_0(\varphi)$, one has 
\begin{equation} \label{DesiredBound} 
\Lip_\Berk(\varphi\vert_I) \ \le \ \max\Big( \frac{1}{\GIR(\varphi) \cdot B_0(\varphi)^d} \, , 
                                  \frac{d}{\GIR(\varphi)^{1/d} \cdot B_0(\varphi)} \Big) \ .
\end{equation} 

First suppose $I = [\alpha,\xi]$.  By the definition of the ball-mapping radius $B_0 = B_0(\varphi)$, 
$\varphi(\cB(\alpha,B_0)^-)$ is a ball, and hence omits some $\beta \in \PP^1(K)$.  
Take any $\gamma_1 \in \GL_2(\cO)$ with $\gamma_2(\beta) = \infty$, and take any $\gamma_2 \in \GL_2(\cO)$ 
with $\gamma_2(0) = \alpha$.   
Put $\Phi = \gamma_1 \circ \varphi \circ \gamma_2$.  Then  $[\alpha,\xi] = \gamma_2([0,\zeta_{0,B_0}])$, 
$\Lip_\Berk(\varphi\vert_{[\alpha,\xi]}) =\Lip_\Berk(\Phi\vert_{[0,\zeta_{0,B_0}]})$, 
and $\Phi(\cD(0,B_0)^-)$ is a disc $\cD(c_0,R)^-$ for some $c_0 \in K$ and some $0 < R < \infty$. 
Propositions \ref{Case1}, \ref{Case2}, \ref{Case3}, and \ref{Case4}  cover all possibilities for $|c_0|$ and $R$, 
and they show that (\ref{DesiredBound}) holds.

Next suppose $I = [\xi,\zeta_G]$.  Take any type II point $\xi_0 \in (\xi,\zeta_G)$, and let $\alpha_0 \in \PP^1(K)$
be such that $\xi_0 \in [\alpha_0,\zeta_G]$.  Choose any $\gamma_2 \in \GL_2(\cO)$ with $\gamma_2(0) = \alpha_0$;  
then $[\xi_0,\zeta_G] \subset \gamma_2([\zeta_{0,B_0},\zeta_G])$. Put $\Phi = \varphi \circ \gamma_2$.  
Then $\Lip_\Berk(\varphi\vert_{[\xi_0,\zeta_G]}) \le \Lip_\Berk(\Phi\vert_{[\zeta_{0,B_0},\zeta_G]})$, and  
Proposition \ref{Case5} shows that $\Lip_\Berk(\Phi\vert_{[\zeta_{0,B_0},\zeta_G]})$ satisfies (\ref{DesiredBound}).
Since we can choose $\xi_0$ as close to $\xi$ as desired, $\Lip_\Berk(\varphi\vert_{[\xi,\zeta_G]})$ satisfies  
(\ref{DesiredBound}) as well. 
\end{proof} 

\begin{proof}[Proof of Theorem \ref{Classical_Lip}.]
Given $\varphi$, we first show that 
\begin{equation} \label{GIRsup} 
\sup_{\substack{x, y \in \PP^1(K) \\ x \ne y}} \frac{\|\varphi(x),\varphi(y)\|}{\|x,y\|} \ \le \ \frac{1}{\GPR(\varphi)} \ . 
\end{equation}

Fix $x, y \in \PP^1(K)$ with $x \ne y$. We claim that $\|\varphi(x),\varphi(y)\|/\|x,y\| \le 1/\GPR(\varphi)$.
If $\|x,y\| \ge \GPR(\varphi)$, the inequality is trivial since $\|\varphi(x),\varphi(y)\| \le 1$.   
Suppose $\|x,y\| <  \GPR(\varphi)$.  After pre-composing and post-composing $\varphi$ with suitable elements 
of $\GL_2(\cO)$, we can assume that $y = 0$ and $\varphi(y) = 0$.  Put $R = \GPR(\varphi)$.  By the definition
of $\GPR(\varphi)$, the image $\varphi(\cD(0,R)^-)$ omits $\zeta_G$.  
In particular, $\varphi$ has no poles in $D(0,R)^-$ and $|\varphi(z)| < 1$ for all $z \in D(0,R)^-$.  
Thus we can expand $\varphi(z)$ as a power series converging in $D(0,R)^-$,
\begin{equation*}
\varphi(z) \ = \ \sum_{i=0}^\infty c_i z^i \ ,
\end{equation*}
where $c_0 = 0$ and $|c_i| \le 1/R^i$ for $i \ge 1$.  Note that $\|x,y\| = |x-0| = |x|$, and that 
\begin{equation*}
\|\varphi(x),\varphi(y)\| \ = \ |\varphi(x)-0| \ = \ |\sum_{i=1}^\infty c_i x^i| 
\ \le \ \max_{i \ge 1} (|x|/R)^i \ = \ |x|/R \ . 
\end{equation*} 
It follows that $\|\varphi(x),\varphi(y)\|/\|x,y\| \le 1/R = 1/\GPR(\varphi)$. 

To complete the proof, we will show that there exist $x, y \in \PP^1(K)$ with 
\begin{equation*}
\frac{\|\varphi(x),\varphi(y)\|}{\|x,y\|} \  = \  \frac{1}{\GPR(\varphi)} \ .  
\end{equation*}   
Let $Q \in \varphi^{-1}(\zeta_G)$ be a point
(necessarily of type II) for which $\diam_G(Q) = \GPR(\varphi)$.  After post-composing $\varphi$ with a 
suitable element of $\GL_2(\cO)$, we can assume $Q = \zeta_{0,R}$;  by construction, $\varphi(Q) = \zeta_G$.
Consider the tangent space $T_Q$.  For any $\vw \in T_{\zeta_G}$, there are at most $d$ directions 
$\vv \in T_Q$ with $\varphi_*(\vv) = \vw$.  Also, there are only finitely many directions $\vv \in T_Q$ for which 
$\varphi(\cB(Q,\vv)^-) = \BPP^1_K$.  Hence, we can choose $\vv_1, \vv_2 \in T_Q \backslash \{\vv_\infty\}$ 
with $\varphi_*(\vv_1) \ne \varphi_*(\vv_2)$, such that $\varphi(\cB(Q,\vv_1)^-)$ and $\varphi(\cB(Q,\vv_2)^-)$ are balls.  
Take any $x \in \PP^1(K) \cap \cB(Q,\vv_1)^-$, $y \in \PP^1(K) \cap \cB(Q,\vv_2)^-$.  
Then $\|x,y\| = R$ and $\|\varphi(x),\varphi(y)\| = 1$.     
\end{proof} 

\begin{proof}[Proof of Theorem \ref{FirstCor}.]
In Theorem \ref{Classical_Lip} we have shown that $\Lip_{\PP^1(K)}(\varphi) \le 1/\GPR(\varphi)$. 
It follows from Corollary \ref{ResultantBoundCor} that $1/|\GPR(\varphi)| \le 1/|\Res(\varphi)|$, 
so $\Lip_{\PP^1(K)}(\varphi) \le 1/|\Res(\varphi)|$.

In Theorem \ref{MainThm0} we have shown that 
\begin{equation} \label{MainBound1}
\Lip_\Berk(\varphi) \ \le \ \max\Big( \frac{d}{\GIR(\varphi)^{1/d} \cdot B_0(\varphi)} \, ,  
                                         \frac{1}{\GIR(\varphi) \cdot  B_0(\varphi)^d} \Big) \ .
\end{equation} 
By Corollary  \ref{ResultantBoundCor} we have $\GIR(\varphi)^d \cdot B_0(\varphi) \ge |\Res(\varphi)|$.  Since  
$1 \ge \GIR(\varphi) > 0$ it follows that $\GIR(\varphi)^{1/d} \ge \GIR(\varphi)^d$, which yields 
$\GIR(\varphi)^{1/d} \cdot B_0(\varphi) \ge |\Res(\varphi)|$.  Similarly 
$\GIR(\varphi)\cdot B_0(\varphi) \ge |\Res(\varphi)|$ and $B_0(\varphi) \ge |\Res(\varphi)|$, so 
$\GIR(\varphi) \cdot  B_0(\varphi)^d \ge |\Res(\varphi)|^d$.  Thus 
$\Lip_\Berk(\varphi) \ \le \ \max\big( d/|\Res(\varphi)|, 1/|\Res(\varphi)|^d\big)$.
\end{proof}

\section{Examples} \label{ExamplesSection}

An analysis of the proof of Theorem \ref{MainThm} leads to the following examples, 
which show the bound (\ref{MainBound}) in Theorem \ref{MainThm} is nearly optimal.  

\medskip
\noindent{\bf Example 1.}  Let $2 \le d \in \ZZ$, and fix $S \in |K^{\times}|$ with $0 < S \le 1$. 
Choose $\beta_1, \ldots, \beta_d \in K$
with $|\beta_i| = S$ for all $i$ and $|\beta_i - \beta_j| = S$ for all $i \ne j$.   
Fix an integer $1 \le k \le d-1$ and a constant $C \in K$ with $|C| \ge 1$, and put 
\begin{equation*}
\varphi(z) \ = \ \frac{Cz^k}{(z-\beta_1) \cdots (z-\beta_d)} \ .
\end{equation*} 
One sees easily that $\varphi(\zeta_G) = \zeta_{0,|C|}$, so $\GIR(\varphi) = 1/|C|$.  
Using the theory of Newton polygons, one sees that for each $a \in \PP^1(K)$,
the image of $B(a,S)^-$ omits at least one point of $\PP^1(K)$, 
and that $\varphi(\cD(0,S)^-) = \cD(0,|C|/S^d)^-$ but $\varphi(\cD(0,S)) = \PP^1(K)$.  
Thus $B_0(\varphi) = S$. For $0 \le r \le S$ one has 
\begin{equation*}
f_\varphi(r) \ := \ \diam_\infty(\varphi(\zeta_{0,r})) \ = \ |C| r^k/S^d \ .
\end{equation*} 
Put $r_1 = (S^d/|C|)^{1/k} \le S$;  then $f_\varphi(r_1) = 1$, and 
\begin{equation*}
\Lip_\Berk(\varphi) \ \ge \ f_\varphi^{\prime}(r_1) \ = \ k \cdot |C| r_1^{k-1}
\ = \ \frac{k}{r_1} \ = \ \frac{k}{\GIR(\varphi)^{1/k} \cdot B_0(\varphi)^{d/k}} \ .
\end{equation*}  
Taking $k = 1$, one sees that the first term in (\ref{MainBound}) cannot be improved.  
Taking $k = d-1$, one obtains a quantity which differs from the second term 
$d/\big(\GIR(\varphi)^{1/d} \cdot B_0(\varphi)\big)$ by a factor 
$\Delta$ satisfying $(d-1)/d < \Delta < 1$.

\medskip
\noindent{\bf Example 2.} With $d$, $S$, $C$ and the $\beta_i$ as in Example 1, put 
\begin{equation*}
\varphi(z) \ = \ \frac{C z^d}{(z-\beta_1) \cdots (z-\beta_{d-1})} \ .
\end{equation*}
As before, one has $\GIR(\varphi) = 1/|C|$ and $B_0(\varphi) = S$.  For $0 \le r \le S$ one has 
\begin{equation*}
f_\varphi(r) \ := \ \diam_\infty(\varphi(\zeta_{0,r})) \ = \ |C| r^d/S^{d-1} \ .
\end{equation*} 
Put $r_1 = (S^{d-1}/|C|)^{1/d} \le S$;  then $f_\varphi(r_1) = 1$, and 
\begin{equation*}
\Lip_\Berk(\varphi) \ \ge \ f_\varphi^{\prime}(r_1) \ = \ d \cdot |C| r_1^{d-1}
\ = \ \frac{d}{r_1} \ = \ \frac{d}{\GIR(\varphi)^{1/d} \cdot B_0(\varphi)^{(d-1)/d}} \ ,
\end{equation*}  
which differs from $d/\big(\GIR(\varphi)^{1/d} \cdot B_0(\varphi)\big)$ by the factor 
$\Delta = B_0(\varphi)^{1/d}$.  

Thus the second term in $(\ref{MainBound})$ cannot be greatly improved, and when $B_0(\varphi) = 1$ it is sharp.

\end{document}